\documentclass[runningheads]{llncs}
\usepackage[T1]{fontenc}
\usepackage{graphicx}

\usepackage{hyperref}       
\usepackage{url}            
\usepackage{booktabs}       
\usepackage{amsfonts}       
\usepackage{nicefrac}       
\usepackage{microtype}      

\usepackage{amssymb, latexsym}
\usepackage{url}
\usepackage{algorithm}
\usepackage{algorithmic}
\usepackage{tabularx}
\usepackage{paralist}
\usepackage{mathtools}

\usepackage{bbm} 

\usepackage{makecell}
\usepackage{multirow}
\usepackage{booktabs}

\usepackage{nicefrac}       

\usepackage[flushleft]{threeparttable} 

\usepackage{caption}
\usepackage{subcaption}
\usepackage{multirow}
\usepackage{colortbl}
\definecolor{bgcolor}{rgb}{0.8,1,1}
\definecolor{bgcolor2}{rgb}{0.8,1,0.8}
\definecolor{niceblue}{rgb}{0.0,0.19,0.56}

\hypersetup{colorlinks,linkcolor={blue},citecolor={blue},urlcolor={blue}}

\usepackage[dvipsnames]{xcolor}
\usepackage{pifont}

\newcommand{\R}{\mathbb{R}}
\newcommand{\eqdef}{\stackrel{\text{def}}{=}}

\def\<#1,#2>{\left\langle #1,#2\right\rangle}

\newtheorem{assumption}{Assumption}

\usepackage{xspace}

\newcommand{\algname}[1]{{\sf  #1}\xspace}

\usepackage[colorinlistoftodos,bordercolor=orange,backgroundcolor=orange!20,linecolor=orange,textsize=scriptsize]{todonotes}

\newcommand{\cB}{{\cal B}}

\newcommand{\cG}{{\cal G}}

\newcommand{\cO}{{\cal O}}

\newcommand{\mA}{{\bf A}}

\newcommand{\EE}{\mathbb{E}}

\newcommand{\NN}{\mathbb{N}}

\begin{document}
\title{Byzantine-Robust Loopless Stochastic Variance-Reduced Gradient
}
%
%
\author{Nikita Fedin\inst{1}
\and
Eduard Gorbunov\inst{2}
}
\authorrunning{N.~Fedin \& E.~Gorbunov}
%
\institute{Moscow Institute of Physics and Technology, Russia\\ \email{fedin.ng@phystech.edu
} \and
Mohamed bin Zayed University of Artificial Intelligence, UAE
\email{eduard.gorbunov@mbzuai.ac.ae}}
\maketitle              
\begin{abstract}
Distributed optimization with open collaboration is a popular field since it provides an opportunity for small groups / companies / universities, and individuals to jointly solve huge-scale problems. However, standard optimization algorithms are fragile in such settings due to the possible presence of so-called Byzantine workers -- participants that can send (intentionally or not) incorrect information instead of the one prescribed by the protocol (e.g., send anti-gradient instead of stochastic gradients). Thus, the problem of designing distributed methods with provable robustness to Byzantine workers has been receiving a lot of attention recently. In particular, several works consider a very promising way to achieve Byzantine tolerance via exploiting variance reduction and robust aggregation. The existing approaches use \algname{SAGA}- and \algname{SARAH}-type variance reduced estimators, while another popular estimator -- \algname{SVRG} -- is not studied in the context of Byzantine-robustness. In this work, we close this gap in the literature and propose a new method -- Byzantine-Robust Loopless Stochastic Variance Reduced Gradient (\algname{BR-LSVRG}). We derive non-asymptotic convergence guarantees for the new method in the strongly convex case and compare its performance with existing approaches in numerical experiments.

\end{abstract}
\section{Introduction}
In this work, we consider a finite-sum minimization problem
\begin{equation}
    \min\limits_{x \in \R^d}\left\{f(x) = \frac{1}{m} \sum\limits_{j=1}^m f_j(x)\right\}. \label{eq:main_problem}
\end{equation}
Such problem formulations are very typical for machine learning tasks \cite{shalev2014understanding,goodfellow2016deep}, where $x$ represents the model parameters and $f_j(x)$ denotes the loss on the $j$-th element of the dataset. In modern problems of this type, the dataset size $m$ and the dimension of the problem $d$ are typically very large, e.g., several billion \cite{ouyang2022training}. Training of such models on $1$ (even very powerful) machine can take years of computations \cite{lidemystifying}. Therefore, it is inevitable to use distributed (stochastic) approaches to solve such complicated problems, e.g., Parallel Stochastic Gradient Descent (\algname{Parallel-SGD}) \cite{robbins1951stochastic,zinkevich2010parallelized}.

Distributed optimization is associated with a number of difficulties related to communication efficiency, data privacy, asynchronous updates, and many other aspects that depend on the setup. One of such aspects is the robustness to \emph{Byzantine workers}\footnote{This term takes its origin in \cite{lamport1982byzantine} and has become standard in the literature \cite{lyu2022privacy}. By using this term, we do not want to offend any group of people but rather follow standard notation for the community.} -- the workers that can (intentionally or not) deviate from the prescribed protocol and are assumed to be omniscient (see more details in Section~\ref{sec:technical_prelim}). Byzantine workers can easily destroy the convergence of standard methods based on simple averaging since such workers can send to the server arbitrary vectors. This fact justifies the usage of special methods that are robust to Byzantine attacks.

In particular, one of the existing techniques to achieve Byzantine-robustness is based on the variance reduction mechanism \cite{gower2020variance}. The key idea behind this approach is based on the fact that variance reduction of stochastic gradients received from regular workers reduces the strengths of Byzantine attacks since it becomes harder to ``hide in the noise'' for Byzantine workers and easier for the server to reduce the effect of Byzantine attacks. This idea led to the development of such variance-reduced Byzantine-robust methods as \algname{Byrd-SAGA} \cite{wu2020federated}, which uses celebrated \algname{SAGA} estimator \cite{defazio2014saga} and geometric median for aggregation, and \algname{Byz-VR-MARINA} \cite{gorbunov2022variance}, which is based on \algname{SARAH} estimator \cite{nguyen2017sarah} and any agnostic robust aggregation \cite{karimireddy2021byzantine}. However, there exists no Byzantine-robust version of another popular variance-reduced method called -- Stochastic Variance-Reduced Gradient (\algname{SVRG}) \cite{johnson2013accelerating}. Moreover, in view of the vulnerability of geometric median to special Byzantine attacks \cite{baruch2019little,xie2020fall,karimireddy2021learning}, it remains unclear how unbiased variance-reduced estimators (like \algname{SVRG}/\algname{SAGA}-estimators) behave in combination with provably robust aggregation rules \cite{karimireddy2021learning,karimireddy2021byzantine} -- the authors \cite{gorbunov2022variance} focus on biased variance reduction only.

\noindent \textbf{Contributions.} We propose a new method called Byzantine-Robust Loopless Stochastic Variance-Reduced Gradient (\algname{BR-LSVRG}) that uses \algname{SVRG}-estimator and (provably) robust aggregation rule. We analyze the method for solving smooth strongly convex distributed optimization problems and prove its theoretical convergence. Though our results require the usage of large enough batchsizes, we show that in certain scenarios \algname{BR-LSVRG} has better convergence guarantees than both \algname{Byrd-SAGA} and \algname{Byz-VR-MARINA}. In addition, we study the convergence of \algname{BR-LSVRG} in several numerical experiments and observe that (i) \algname{BR-LSVRG} can reach a good accuracy of the solution even with small batchsize and (ii) \algname{BR-LSVRG} converges better than \algname{Byrd-SAGA}.

\subsection{Technical Preliminaries}\label{sec:technical_prelim}

\paragraph{Notation.} We denote the standard Euclidean inner product in $\R^d$ as $\langle x, y \rangle \eqdef \sum_{i=1}^d x_iy_i$, where $x = (x_1,\ldots,x_d)^\top, y = (y_1,\ldots,y_d)^\top \in \R^d$ and $\ell_2$-norm as $\|x\| \eqdef \sqrt{\langle x, x \rangle}$. For any integer $t > 0$ we use $[t]$ to define set $\{1,2,\ldots,t\}$. Finally, $\EE[\cdot]$ denotes full expectation and $\EE_k[\cdot]$ denotes the expectation w.r.t.\ the randomness coming from iteration $k$.

\paragraph{Byzantine workers.} We assume that the distributed system consists of $n$ workers $[n]$ connected with parameter-server. Each worker can compute gradients of $\nabla f_j(x)$ for any $j \in [m]$ and $x \in \R^d$. Moreover, we assume that workers consist of two groups: $[n] = \cG \cup \cB$, $\cG \cap \cB = \varnothing$. Here $\cG$ denotes the set of regular workers, and $\cB$ is the set of so-called \emph{Byzantine workers}, i.e., the workers that can (intentionally or not) send arbitrary vectors to the server instead of ones prescribed by the algorithm. Moreover, following the classical convention, we assume that Byzantine workers can be omniscient, meaning that they can know exactly what other workers send to the server and what aggregation rule the server uses. Although this assumption is strong, it is quite popular due to the following argument: if the method is robust to the presence of Byzantine workers in these settings, this method is guaranteed to be robust in scenarios when Byzantine workers are less harmful. In addition, one has to assume that $|\cG| = G \geq (1-\delta)n$ or equivalently $|\cB| = B \leq \delta n$ for some $\delta < \nicefrac{1}{2}$ (otherwise Byzantine workers form a majority and provable Byzantine-robustness cannot be achieved in the worst case \cite{karimireddy2021byzantine}).

\paragraph{Robust aggregation.} Following \cite{karimireddy2021byzantine,gorbunov2022variance}, we use the following definition of the robust aggregator/aggregation.
\begin{definition}[$(\delta, c)$-robust aggregator \cite{karimireddy2021byzantine,gorbunov2022variance}]\label{def:robust_aggr}
    Let $x_1,x_2,\ldots, x_n$ be such that for some subset $\cG \subseteq [n]$ of size $|\cG| = G \geq (1-\delta)n$, $\delta < \nicefrac{1}{2}$ there exists $\sigma \geq 0$ such that $\frac{1}{G(G-1)}\sum_{i,l \in \cG}\EE\|x_i - x_l\|^2 \leq \sigma^2$. The quantity $\widehat x$ is called $(\delta, c)$-robust aggregator ($(\delta, c)$-\texttt{RAgg}) and denoted as $\widehat x = \texttt{RAgg}(x_1,\ldots, x_n)$ for some number $c > 0$ if the following holds:
    \begin{equation}
        \EE\|\widehat x - \overline x\|^2 \leq c\delta\sigma^2, \label{eq:robust_aggr}
    \end{equation}
    where $\overline{x} = \frac{1}{G}\sum_{i\in \cG} x_i$. In addition, if $\widehat x$ can be computed without the knowledge of $\sigma^2$, then $\widehat x$ is called $(\delta, c)$-agnostic robust aggregator ($(\delta, c)$-\texttt{ARAgg}) and denoted as $\widehat x = \texttt{ARAgg}(x_1,\ldots, x_n)$.
\end{definition}
In other words, the aggregator $\widehat x$ is called robust if, on average, it is ``not far'' from $\overline{x}$ -- the average of the vectors from regular workers $\cG$. Here, the upper bound on how far we allow the robust aggregator to be from the average over regular workers depends on the variance of regular workers and the ratio of Byzantines. It is relatively natural that both characteristics should affect the quality of the aggregation. Moreover, there exists a lower bound stating that for any aggregation rule $\widehat x$ there exists a set of vectors satisfying the conditions of the above definition such that $\EE\|\widehat x - \overline x\|^2 = \Omega(\delta \sigma^2)$, which formally establishes the tightness of Definition~\ref{def:robust_aggr}. We provide several examples of robust aggregators in Appendix~\ref{appendix:robust_aggregation}.

\paragraph{Assumptions.} We make a standard assumption for the analysis of variance-reduced methods \cite{kovalev2020don}.
\begin{assumption}\label{as:main}
    Functions $f_1,\ldots, f_m:\R^d \to \R$ are convex and $L$-smooth ($L > 0$), and function $f:\R^d \to \R$ is additionally $\mu$-strongly convex ($\mu > 0$), i.e., for all $x,y \in \R^d$
    \begin{gather}
        f_j(y) \geq f_j(x) + \langle \nabla f_j(x), y - x \rangle \quad \forall j\in [m], \label{eq:f_j_convex} \\
        \|\nabla f_j(x) - \nabla f_j(y)\| \leq L\|x - y\| \quad \forall j\in [m], \label{eq:f_j_L_smooth}\\
        f(y) \geq f(x) + \langle \nabla f(x), y - x \rangle + \frac{\mu}{2}\|y - x\|^2. \label{eq:f_mu_strongly_convex}
    \end{gather}
\end{assumption}

For one of our results, we make the following additional assumption, which is standard for the stochastic optimization literature \cite{nemirovski2009robust,ghadimi2013stochastic}.
\begin{assumption}\label{as:UBV}
    We assume that there exists $\sigma \geq 0$ such that for any $x \in \R^d$ and $j$ being sampled uniformly at random from $[m]$
    \begin{equation}
        \EE\|\nabla f_j(x) - \nabla f(x)\|^2 \leq \sigma^2. \label{eq:UBV}
    \end{equation}
\end{assumption}

\subsection{Related Work}

Many existing methods for Byzantine-robust distributed optimization are based on the replacement of averaging with special aggregation rules in \algname{Parallel-SGD} \cite{blanchard2017machine,yin2018byzantine,damaskinos2019aggregathor,guerraoui2018hidden,pillutla2022robust}. As it is shown in \cite{baruch2019little,xie2020fall}, such approaches are not Byzantine-robust and can even perform worse than na\"ive \algname{Parallel-SGD} for particular Byzantine-attacks. To circumvent this issue, the authors of \cite{karimireddy2021learning} introduce a formal definition of robust aggregator (see Definition~\ref{def:robust_aggr}) and propose the first distributed methods with provable Byzantine-robustness. The key ingredient in their method is client heavy-ball-type momentum \cite{polyak1964some} to make the method non-permutation-invariant that prevents the algorithm from time-coupled Byzantine attacks. In \cite{karimireddy2021byzantine}, this technique is generalized to heterogeneous problems and robust aggregation agnostic to the noise. An extension to decentralized optimization problems is proposed by \cite{he2022byzantine}. Another approach that ensures Byzantine-robustness both in theory and practice is based on the checks of computations at random moments of time \cite{gorbunov2022secure}. Finally, there are two approaches based on variance reduction\footnote{See \cite{gower2020variance} for a recent survey on variance-reduced methods.} mechanism -- \algname{Byrd-SAGA} \cite{wu2020federated}, which uses well-suited for convex problems \algname{SAGA}-estimator \cite{nguyen2017sarah,fang2018spider,horvath2022adaptivity,li2021page}, and \algname{Byz-VR-MARINA} \cite{gorbunov2022variance}, which employs well-suited for non-convex problems \algname{SARAH}-estimator. We refer to \cite{lyu2022privacy,gorbunov2022secure} for the extensive summaries of other existing approaches.

\section{Main Results}

In this section, we introduce the new method called Byzantine-Robust distributed Loopless Stochastic Variance-Reduced Gradient (\algname{BR-LSVRG}, Algorithm~\ref{alg:BR_LSVRG}). At each iteration of  \algname{BR-LSVRG}, regular workers compute standard \algname{SVRG}-estimator \cite{johnson2013accelerating} (line 7) and send it to the server. The algorithm has two noticeable features. First, unlike many existing distributed methods that use averaging or other aggregation rules vulnerable to Byzantine attacks, \algname{BR-LSVRG} uses a provably robust aggregator (according to Definition~\ref{def:robust_aggr}) on the server. Secondly, following the idea of \cite{hofmann2015variance,kovalev2020don}, in \algname{BR-LSVRG}, regular workers update the reference point $w_i^{k+1}$ as $x^k$ with some small probability $p$. When $w_i^{k+1} = x^k$, worker $i$ has to compute the full gradient during the next step in order to calculate $g_i^k$; otherwise, only $2b$ gradients of the summands from \eqref{eq:main_problem} need to be computed. To make the expected computation cost (number of computed gradients of the summands from \eqref{eq:main_problem}) of $1$ iteration on each regular worker to be $\cO(b)$, probability $p$ is chosen as $p \sim \nicefrac{b}{m}$.

\begin{algorithm}[t]
\caption{Byzantine-Robust Distributed \algname{LSVRG} (\algname{BR-LSVRG})}\label{alg:BR_LSVRG}
\begin{algorithmic}[1]
\STATE {\bfseries Input:} stepsize $\gamma > 0$, batchsize $b \geq 1$, starting point $x^0$, probability $p \in (0, 1]$, $(\delta,c)$-agnostic robust aggregator $\texttt{ARAgg}(x_1, \ldots, x_n)$, number of iterations $K > 0$
\STATE Set $w_i^0 = x^0$ and compute $\nabla f(w_i^0)$ for all $i\in \cG$
\FOR {$k = 0, 1, \ldots , K-1$}
    \STATE Server sends $x^k$ to all workers
    \FOR {all $i \in \cG$}
        \STATE Choose $j_{i,k}^1, j_{i,k}^2, \ldots, j_{i,k}^b$ from $[m]$ uniformly at random independently from each other and other workers
        \STATE $g_i^k = \frac{1}{b}\sum\limits_{t=1}^b\left(\nabla f_{j_{i,k}^t}(x^k) - \nabla f_{j_{i,k}^t}(w_i^k)\right) + \nabla f(w_i^k)$
        \STATE $w_i^{k+1} = \begin{cases} x^k,& \text{with probability } p \\ w_i^k,& \text{with probability } 1-p \end{cases}$
        \STATE Send $g_i^k$ to the server
    \ENDFOR
    \FOR {all $i \in \cB$}
        \STATE Send $g_i^k = *$ (anything) to the server
    \ENDFOR
    \STATE Server receives vectors $g_1^k, \ldots, g_n^k$ from the workers
    \STATE Server computes $x^{k+1} = x^k - \gamma \cdot \texttt{ARAgg}(g_1^k,\ldots, g_n^k)$
\ENDFOR
\STATE {\bfseries Output:} $x^K$
\end{algorithmic}
\end{algorithm}

We start the theoretical convergence analysis with the following result.

\begin{theorem}\label{thm:convergece_under_UBV_assumption}
    Let Assumptions~\ref{as:main} and \ref{as:UBV} hold, batchsize $b \geq 1$, and stepsize $0 < \gamma \leq \min\left\{\nicefrac{1}{12L}, \nicefrac{p}{\mu}\right\}$. Then, the iterates produced by \algname{BR-LSVRG} after $K$ iterations satisfy
    \begin{equation}
        \EE\Psi_{K} \leq \left(1 - \frac{\gamma\mu}{2}\right)^{K}\Psi_0 + \gamma\frac{32c\delta \sigma^2}{b\mu} + \frac{32c\delta \sigma^2}{b\mu^2}, \label{eq:conv_UBV}
    \end{equation}
    where $\Psi_k \eqdef \|x^k - x^*\|^2  + \tfrac{8\gamma^2}{p}\sigma_k^2$, $\sigma_k^2 \eqdef \tfrac{1}{Gm}\sum_{i\in \cG}\sum_{j=1}^m\|\nabla f_j(w_{i}^k) - \nabla f_j(x^*)\|^2$.
\end{theorem}
\begin{proof} For convenience, we introduce new notation: $\widehat g^k = \texttt{ARAgg}(g_1^k,\ldots, g_n^k)$. Then, $x^{k+1} = x^k - \gamma \widehat g^k$ and
\begin{eqnarray*}
    \|x^{k+1} - x^*\|^2 
    &=& \|x^k - x^*\|^2 - 2\gamma \langle x^k-x^*, \overline{g}^k \rangle - 2\gamma \langle x^k -x^* , \widehat{g}^k - \overline{g}^k \rangle\\
    &&  + \gamma^2 \|\overline{g}^k + (\widehat{g}^k - \overline{g}^k) \|^2.
\end{eqnarray*}
Next, we apply inequalities $\langle a,b \rangle \leq \frac{\alpha}{2}\|a\|^2 + \frac{1}{2\alpha}\|b\|^2$ and $\|a+b\|^2 \leq 2\|a\|^2 + 2\|b\|^2$, which hold for any $a,b \in \R^d$, $\alpha > 0$, to the last two terms and take expectation $\EE_k[\cdot]$ from both sides of the above inequality (note that $\EE_k[\overline{g}^k] = \nabla f(x^k)$)
\begin{eqnarray}
    \EE_k\|x^{k+1} - x^*\|^2 &\leq& \left(1 + \frac{\gamma\mu}{2}\right)\|x^k - x^*\|^2 - 2\gamma \langle x^k-x^*, \nabla f(x^k)\rangle  + 2\gamma^2 \EE_k \|\overline{g}^k\|^2 \notag \\
    && +2\gamma \left(\frac{1}{\mu} + \gamma\right) \EE_k \|\widehat{g}^k - \overline{g}^k\|^2\notag \\
    &\overset{\eqref{eq:f_mu_strongly_convex}}{\leq}& 
    \left(1 - \frac{\gamma\mu}{2}\right)\|x^k - x^*\|^2 - 2\gamma\left(f(x^k) - f(x^*)\right)  \notag\\
    && + 2\gamma^2 \EE_k \|\overline{g}^k\|^2 +2\gamma \left(\frac{1}{\mu} + \gamma\right) \EE_k \|\widehat{g}^k - \overline{g}^k\|^2.\label{eq:technical_1}
\end{eqnarray}
To proceed with the derivation, we need to upper bound the last two terms from the above inequality. For the first term, we use Jensen's inequality and the well-known fact that the variance is not larger than the second moment:
\begin{eqnarray}
    \EE_k\|\overline{g}^k\|^2 
    &\leq& \frac{1}{bG}\sum\limits_{i \in \cG}\sum\limits_{t=1}^b\EE_k \|\nabla f_{j_{i,k}^t}(x^k) - \nabla f_{j_{i,k}^t}(w_i^k) + \nabla f(w_i^k)\|^2 \notag\\
    &=& \frac{1}{G}\sum\limits_{i \in \cG}\EE_k \|\nabla f_{j_{i,k}^1}(x^k) - \nabla f_{j_{i,k}^1}(w_i^k) + \nabla f(w_i^k)\|^2 \notag\\
    &\leq& \frac{2}{G}\sum\limits_{i \in \cG}\EE_k \|\nabla f_{j_{i,k}}(x^k) - \nabla f_{j_{i,k}}(x^*)\|^2 \notag\\
    &&  +  \frac{2}{G}\sum\limits_{i \in \cG}\EE_k \|\nabla f_{j_{i,k}^1}(w_i^k) - \nabla f_{j_{i,k}^1}(x^*) - \nabla f(w_i^k)\|^2 \notag\\
    &\leq& \frac{2}{G}\sum\limits_{i \in \cG}\EE_k\left[ \|\nabla f_{j_{i,k}^1}(x^k) - \nabla f_{j_{i,k}^1}(x^*)\|^2 + \|\nabla f_{j_{i,k}^1}(w_i^k) - \nabla f_{j_{i,k}^1}(x^*)\|^2\right] \notag\\
    &=& \frac{2}{Gm}\sum\limits_{i \in \cG}\sum\limits_{j=1}^m\left(\|\nabla f_j(x^k) - \nabla f_j(x^*)\|^2 + \|\nabla f_j(w_i^k) - \nabla f_j(x^*)\|^2 \right) \notag\\
    &\overset{(*)}{\leq}& \frac{4L}{Gm}\sum\limits_{i \in \cG}\sum\limits_{j=1}^m\left(f_j(x^k) - f_j(x^*) - \langle \nabla f_j(x^*), x^k - x^* \rangle\right) + 2\sigma_k^2 \notag\\
    &=& 4L\left(f(x^k) - f(x^*)\right) + 2\sigma_k^2, \label{eq:second_moment_bound}
\end{eqnarray}
where in $(*)$ we use that for any convex $L$-smooth function $h(x)$ and any $x,y \in\R^d$ (e.g., see \cite{nesterov2018lectures})
\begin{equation}
    \|\nabla h(x) - \nabla h(y)\|^2 \leq 2L\left(h(x) - h(y) - \langle \nabla h(y), x - y \rangle\right). \label{eq:L_smoothness_corollary}
\end{equation}
To bound the last term from \eqref{eq:technical_1}, we notice that Assumption~\ref{as:UBV} gives $\forall i,l \in \cG$
\begin{align*}
    \EE_k&\|g_i^k - g_l^k\|^2 \\
    &=\EE_k\|g_i^k - \nabla f(x^k)\|^2 + \EE_k\|g_l^k - \nabla f(x^k)\|^2 + 2\EE_k\left[\langle g_i^k - \nabla f(x^k), g_l^k - \nabla f(x^k) \rangle\right]\\
    &= \EE_k \left\|\frac{1}{b}\sum\limits_{t=1}^b(\nabla f_{j_{i,k}^t}(x^k) - \nabla f(x^k) - (\nabla f_{j_{i,k}^t}(w_i^k) - \nabla f(w_i^k)))\right\|^2\\
    &  + \EE_k \left\|\frac{1}{b}\sum\limits_{t=1}^b(\nabla f_{j_{l,k}^t}(x^k) - \nabla f(x^k) - (\nabla f_{j_{l,k}^t}(w_l^k) - \nabla f(w_l^k)))\right\|^2\\
    &\leq 2\EE_k\left[\left\|\frac{1}{b}\sum\limits_{t=1}^b(\nabla f_{j_{i,k}}(x^k) - \nabla f(x^k))\right\|^2 + \left\|\frac{1}{b}\sum\limits_{t=1}^b(\nabla f_{j_{i,k}}(w_i^k) - \nabla f(w_i^k))\right\|^2\right]\\
    &  + 2\EE_k\left[\left\|\frac{1}{b}\sum\limits_{t=1}^b(\nabla f_{j_{l,k}}(x^k) - \nabla f(x^k))\right\|^2 + \left\|\frac{1}{b}\sum\limits_{t=1}^b(\nabla f_{j_{l,k}}(w_l^k) - \nabla f(w_l^k))\right\|^2\right]\\
    &\leq \frac{2}{b^2}\sum\limits_{t=1}^b\EE_k\left[\|\nabla f_{j_{i,k}}(x^k) - \nabla f(x^k)\|^2 + \|\nabla f_{j_{i,k}}(w_i^k) - \nabla f(w_i^k)\|^2\right]\\
    & + \frac{2}{b^2}\sum\limits_{t=1}^b\EE_k\left[\|\nabla f_{j_{l,k}}(x^k) - \nabla f(x^k)\|^2 + \|\nabla f_{j_{l,k}}(w_l^k) - \nabla f(w_l^k)\|^2\right] \overset{\eqref{eq:UBV}}{\leq} \frac{8\sigma^2}{b}.
\end{align*}
Then, $\frac{1}{G(G-1)}\sum\limits_{i,l \in \cG}\EE_k\|g_i^k - g_l^k\|^2 \leq \frac{8\sigma^2}{b}$ 
and by Definition~\ref{def:robust_aggr} we have
\begin{equation}
    \EE_k \|\widehat{g}^k - \overline{g}^k\|^2 \overset{\eqref{eq:robust_aggr}}{\leq} \frac{8c\delta \sigma^2}{b}. \label{eq:distortion_bound}
\end{equation}
Putting all together in \eqref{eq:technical_1}, we arrive at
\begin{eqnarray}
    \EE_k\Psi_{k+1} &\leq& \left(1 - \frac{\gamma\mu}{2}\right)\|x^k - x^*\|^2 + 4\gamma^2 \sigma_k^2 - 2\gamma(1 - 4L\gamma)\left(f(x^k) - f(x^*)\right)\notag\\
    &&  + 16\gamma \left(\frac{1}{\mu} + \gamma\right)\frac{c\delta \sigma^2}{b} + \frac{8\gamma^2}{p}\EE_k[\sigma_{k+1}^2]. \label{eq:squared_distance_bound}
\end{eqnarray}

Next, we estimate $\EE_k[\sigma^2_{k+1}]:$
\begin{eqnarray}
\EE_k[\sigma^2_{k+1}] 
&=& \frac{1-p}{Gm}\sum\limits_{i\in \cG} \sum\limits_{j=1}^m \| \nabla f_j(w_{i}^k) - \nabla f_j(x^*)\|^2   + \frac{p}{m}\sum\limits_{j=1}^m \| \nabla f_j(x^k) - \nabla f_j(x^*)\|^2 \notag\\
&\overset{\eqref{eq:L_smoothness_corollary}}{\leq}& (1-p) \sigma^2_k + \frac{2Lp}{m} \sum\limits_{j=1}^m \left(f_j(x^k) - f_j(x^*) - \langle \nabla f_j(x^*), x^k - x^* \rangle\right)\notag\\
&=& (1-p) \sigma^2_k + 2Lp\left(f(x^k) - f(x^*)\right).\label{eq:sigma_k+1_bound}
\end{eqnarray}

Finally, we combine \eqref{eq:squared_distance_bound} and \eqref{eq:sigma_k+1_bound}:
\begin{eqnarray*}
    \EE_k\Psi_{k+1} 
    &\leq& \left(1 - \frac{\gamma\mu}{2}\right)\|x^k - x^*\|^2 + \left(1 - \frac{p}{2}\right)\frac{8\gamma^2}{p}\sigma_k^2\\
    &&  - 2\gamma (1 - 12L\gamma)\left(f(x^k) - f(x^*)\right) + 16\gamma \left(\frac{1}{\mu} + \gamma\right) \frac{c\delta \sigma^2}{b}\\
    &\leq& \left(1 - \min\left\{\frac{\gamma\mu}{2}, \frac{p}{2}\right\}\right)\Psi_k - 2\gamma (1 - 12L\gamma)\left(f(x^k) - f(x^*)\right)\\
    &&  + 16\gamma \left(\frac{1}{\mu} + \gamma\right) \frac{c\delta \sigma^2}{b}\\
    &\leq& \left(1 - \frac{\gamma\mu}{2}\right)\Psi_k + 16\gamma \left(\frac{1}{\mu} + \gamma\right) \frac{c\delta \sigma^2}{b},
\end{eqnarray*}
where in the last step we use $0 < \gamma \leq \min\left\{\nicefrac{1}{12L}, \nicefrac{p}{\mu}\right\}$. Taking the full expectation and unrolling the obtained recurrence, we get that for all $K \geq 0$
\begin{eqnarray*}
    \EE\Psi_K &\leq& \left(1 - \frac{\gamma\mu}{2}\right)^{K}\Psi_0 + 16\gamma \left(\frac{1}{\mu} + \gamma\right) \frac{c\delta \sigma^2}{b}\sum\limits_{k=0}^{K-1}\left(1 - \frac{\gamma\mu}{2}\right)^{k}\\
    &\leq& \left(1 - \frac{\gamma\mu}{2}\right)^{K}\Psi_0 + \gamma\frac{32c\delta \sigma^2}{b\mu} + \frac{32c\delta \sigma^2}{b\mu^2},
\end{eqnarray*}
which concludes the proof. $\square$
\end{proof}

Since $\Psi_K \geq \|x^K - x^*\|^2$, Theorem~\ref{thm:convergece_under_UBV_assumption} states that \algname{BR-LSVRG} converges linearly (in expectation) to the neighborhood of the solution. We notice that the neighborhood's size is proportional to $\nicefrac{\sigma^2}{b}$, which is typical for Stochastic Gradient Descent-type methods \cite{gorbunov2020unified}, and also proportional to the ratio of Byzantine workers $\delta$. When $\delta = 0$, \algname{BR-LSVRG} converges linearly and recovers the rate of \algname{LSVRG} \cite{kovalev2020don} up to numerical factors. However, in the general case ($\delta > 0$), the last term in \eqref{eq:conv_UBV} can be reduced only via increasing batchsize $b$. We believe that this is unavoidable for \algname{BR-LSVRG} in the worst case since robust aggregation creates a bias in the update, and the analysis of \algname{LSVRG} is very sensitive to the bias in the update direction. Nevertheless, when the size of the neighborhood is small enough, e.g., when $\delta$ or $\sigma^2$ are small, the method can achieve relatively good accuracy with moderate batchsize -- we demonstrate this phenomenon in the experiments.

Next, we present an alternative convergence result that does not rely on the bounded variance assumption.
\begin{theorem}\label{thm:no_UBV}
   Let Assumption~\ref{as:main} hold, $0 < \gamma \leq \min\left\{\nicefrac{1}{144L}, \nicefrac{p}{\mu}\right\}$, and $b \geq \max\left\{1, \nicefrac{c\delta}{\gamma\mu}\right\}$. Then, the iterates of \algname{BR-LSVRG} after $K$ iterations satisfy
    \begin{equation*}
        \EE\Psi_{K} \leq \left(1 - \frac{\gamma\mu}{2}\right)^{K}\Psi_0,
    \end{equation*}
    where $\Psi_k \eqdef \|x^k - x^*\|^2  + \tfrac{72\gamma^2}{p}\sigma_k^2$, $\sigma_k^2 \eqdef \tfrac{1}{Gm}\sum_{i\in \cG}\sum_{j=1}^m\|\nabla f_j(w_{i}^k) - \nabla f_j(x^*)\|^2$. 
\end{theorem}
\begin{proof}
    First, we notice that inequalities \eqref{eq:technical_1}, \eqref{eq:second_moment_bound}, and \eqref{eq:sigma_k+1_bound} from the proof of Theorem~\ref{thm:convergece_under_UBV_assumption} are derived without Assumption~\ref{as:UBV}. We need to derive a version of \eqref{eq:distortion_bound} that does not rely on Assumption~\ref{as:UBV}. Due to the independence of $j_{i,k}^1, j_{i,k}^2, \ldots, j_{i,k}^b$ we have $\forall i \in \cG$
    \begin{eqnarray}
        \EE_k\|g_i^k - \nabla f(x^k)\|^2 
        &=& \frac{1}{b^2}\sum\limits_{t=1}^b\EE_k\|\nabla f_{j_{i,k}^t}(x^k) - \nabla f_{j_{i,k}^t}(w_i^k) + \nabla f(w_i^k) - \nabla f(x^k)\|^2 \notag\\
        &\leq& \frac{1}{b^2}\sum\limits_{t=1}^b\EE_k\|\nabla f_{j_{i,k}^t}(x^k) - \nabla f_{j_{i,k}^t}(w_i^k)\|^2 \notag
    \end{eqnarray}
    and
    \begin{align}
        \EE_k&\|g_i^k - \nabla f(x^k)\|^2 \notag\\
        &\leq \frac{2}{b^2}\sum\limits_{t=1}^b\EE_k\left[\|\nabla f_{j_{i,k}^t}(x^k) - \nabla f_{j_{i,k}^t}(x^*)\|^2 + \|\nabla f_{j_{i,k}^t}(w_i^k) - \nabla f_{j_{i,k}^t}(x^*)\|^2\right] \notag\\
        &= \frac{2}{bm}\sum\limits_{j=1}^m\left(\|\nabla f_j(x^k) - \nabla f_j(x^*)\|^2 +\|\nabla f_j(w_i^k) - \nabla f_j(x^*)\|^2\right) \notag\\
        &\overset{\eqref{eq:L_smoothness_corollary}}{\leq} \frac{4L}{b}\left(f(x^k) - f(x^*)\right) + \frac{2}{bm}\sum\limits_{j=1}^m\|\nabla f_j(w_i^k) - \nabla f_j(x^*)\|^2. \label{eq:technical_no_UBV_1}
    \end{align}
    Therefore,
    \begin{eqnarray*}
        \frac{1}{G(G-1)}\sum\limits_{i,l \in \cG}\EE_k\|g_i^k - g_l^k\|^2 
        &\leq& \frac{2}{G(G-1)}\sum\limits_{i,l \in \cG, i\neq l}\EE_k\|g_i^k  - \nabla f(x^k)\|^2\\
        &&  + \frac{2}{G(G-1)}\sum\limits_{i,l \in \cG, i\neq l}\EE_k\|g_l^k  - \nabla f(x^k)\|^2 \\
        &=& \frac{4}{G}\sum\limits_{i\in \cG} \EE_k\|g_i^k  - \nabla f(x^k)\|^2\\
        &\overset{\eqref{eq:technical_no_UBV_1}}{\leq}& 
        \frac{16L}{b}\left(f(x^k) - f(x^*)\right) + \frac{8}{b}\sigma_k^2,
    \end{eqnarray*}
    and by definition of $(\delta,c)$-robust aggregator we have $\EE_k \|\widehat{g}^k - \overline{g}^k\|^2 \overset{\eqref{eq:robust_aggr}}{\leq} \frac{16Lc\delta}{b}\left(f(x^k) - f(x^*)\right) + \frac{8c\delta}{b}\sigma_k^2$. 
Combining this inequality with \eqref{eq:technical_1} and \eqref{eq:second_moment_bound}, we get
\begin{eqnarray*}
    \EE_k\|x^{k+1} - x^*\|^2 &\leq& \left(1 - \frac{\gamma\mu}{2}\right)\|x^k - x^*\|^2 - 2\gamma\left(f(x^k) - f(x^*)\right)  \notag\\
    && + 8L\gamma^2 \left(f(x^k) - f(x^*)\right) + 4\gamma^2 \sigma_k^2\\
    &&  + \frac{32\gamma Lc\delta}{b} \left(\frac{1}{\mu} + \gamma\right)\left(f(x^k) - f(x^*)\right) + \frac{16\gamma c\delta}{b}\left(\frac{1}{\mu} + \gamma\right)\sigma_k^2\\
    &\leq& \left(1 - \frac{\gamma\mu}{2}\right)\|x^k - x^*\|^2 + 36\gamma^2 \sigma_k^2\\
    &&  - 2\gamma\left(1 - 72\gamma L\right)\left(f(x^k) - f(x^*)\right),
\end{eqnarray*}
where in the last step we use $b \geq \max\left\{1, \nicefrac{c\delta}{\gamma\mu}\right\}$. Finally, this inequality and \eqref{eq:sigma_k+1_bound} imply that
\begin{eqnarray*}
    \EE_k\Psi_{k+1} 
    &\leq& \left(1 - \frac{\gamma\mu}{2}\right)\|x^k - x^*\|^2 + 36\gamma^2 \sigma_k^2 - 2\gamma\left(1 - 72\gamma L\right)\left(f(x^k) - f(x^*)\right)\\
    &&  + (1 - p)\frac{72\gamma^2}{p}\sigma_k^2 + 144L\gamma^2\left(f(x^k) - f(x^*)\right)\\
    &=& \left(1 - \frac{\gamma\mu}{2}\right)\|x^k - x^*\|^2 + \left(1 - \frac{p}{2}\right)\frac{72\gamma^2}{p}\sigma_k^2\\
    &&  - 2\gamma\left(1 - 144\gamma L\right)\left(f(x^k) - f(x^*)\right) \leq \left(1 - \frac{\gamma\mu}{2}\right)\Psi_k,
\end{eqnarray*}
where the last step follows from $\gamma \leq \min\left\{\nicefrac{1}{144L}, \nicefrac{p}{\mu}\right\}$. Taking the full expectation from both sides and unrolling the recurrence, we get the result. $\square$
\end{proof}

In contrast to Theorem~\ref{thm:convergece_under_UBV_assumption},  Theorem~\ref{thm:no_UBV} establishes linear convergence of \algname{BR-LSVRG} to any accuracy. However, Theorem~\ref{thm:no_UBV} requires batchsize to satisfy $b \geq \max\{1, \nicefrac{c\delta}{\gamma \mu}\}$, which can be huge in the worst case. If $p = \nicefrac{b}{m}$, $\gamma = \min\{\nicefrac{1}{144L}, \nicefrac{b}{m\mu}\}$, $b = \max\{1, \nicefrac{144c\delta L}{\mu}, \sqrt{c\delta m}\}$, and $m \geq b$ (for example, these assumptions are satisfied when $m$ is sufficiently large), then, according to Theorem~\ref{thm:no_UBV}, \algname{BR-LSVRG} finds $x^K$ such that $\EE\|x^K - x^*\|^2 \leq \varepsilon \Psi_0$ after
\begin{gather}
    \cO\left(\left(\frac{L}{\mu} + \frac{m}{b}\right)\log\frac{1}{\varepsilon}\right) \text{ iterations}, \label{eq:BRLSVRG_iteration_complexity}\\
    \cO\left(\left(\frac{L}{\mu} + \frac{L^2\sqrt{c\delta}}{\mu^2} + \frac{L\sqrt{c\delta m}}{\mu} + m\right)\log\frac{1}{\varepsilon}\right) \text{ oracle calls}. \label{eq:BRLSVRG_oracle_complexity}
\end{gather}
Under the same assumptions, to achieve the same goal \algname{Byrd-SAGA} requires \cite{wu2020federated}
\begin{gather}
    \cO\left(\frac{m^2L^2}{b^2(1-2\delta)\mu^2}\log\frac{1}{\varepsilon}\right) \text{ iterations}, \label{eq:Byrd_SAGA_iteration_complexity}\\
    \cO\left(\frac{m^2L^2}{b(1-2\delta)\mu^2}\log\frac{1}{\varepsilon}\right) \text{ oracle calls}. \label{eq:Byrd_SAGA_oracle_complexity}
\end{gather}
Complexity bounds for \algname{Byrd-SAGA} are inferior to the ones derived for \algname{BR-LSVRG} as long as our result is applicable. Moreover, when $\delta = 0$ (no Byzantines)  our result recovers the known one for \algname{LSVRG} (up to numerical factors), while the upper bounds \eqref{eq:Byrd_SAGA_iteration_complexity} and \eqref{eq:Byrd_SAGA_oracle_complexity} are much larger than the best-known ones for \algname{SAGA}. This comparison highlights the benefits of our approach compared to the closest one.

Finally, we compare our results against the current state-of-the-art ones obtained for \algname{Byz-VR-MARINA} in \cite{gorbunov2022variance}. In particular, under weaker conditions (Polyak-{\L}ojasiewicz condition \cite{polyak1963gradient,lojasiewicz1963topological} instead of strong convexity), the authors of \cite{gorbunov2022variance} prove that to achieve $\EE[f(x^K) - f(x^*)] \leq \varepsilon (f(x^0) - f(x^*))$ \algname{Byz-VR-MARINA} requires
\begin{gather}
    \cO\left(\left(\frac{L}{\mu} + \frac{L\sqrt{m}}{\mu b\sqrt{n}} + \frac{L m \sqrt{c\delta}}{\mu \sqrt{b^3}} + \frac{m}{b}\right)\log\frac{1}{\varepsilon}\right) \text{ iterations}, \label{eq:Byz_VR_MARINA_iteration_complexity}\\
    \cO\left(\left(\frac{bL}{\mu} + \frac{L\sqrt{m}}{\mu \sqrt{n}} + \frac{L m \sqrt{c\delta}}{\mu \sqrt{b}} + m\right)\log\frac{1}{\varepsilon}\right) \text{ oracle calls}. \label{eq:Byz_VR_MARINA_oracle_complexity}
\end{gather}
Complexity bounds for \algname{Byz-VR-MARINA} are not better than ones derived for \algname{BR-LSVRG} as long as our result is applicable. Moreover, in the special case, when $m > b\sqrt{n}$ (big data regime), iteration complexity of \algname{Byz-VR-MARINA} \eqref{eq:Byz_VR_MARINA_iteration_complexity} is strictly worse than the one we have for \algname{BR-LSVRG} \eqref{eq:BRLSVRG_iteration_complexity}. When $\delta = 0$, our results are strictly better than the ones for \algname{Byz-VR-MARINA}. However, it is important to notice that (i) the results for \algname{Byz-VR-MARINA} are derived under weaker assumptions and (ii) in contrast to the results for \algname{Byrd-SAGA} and \algname{Byz-VR-MARINA}, our results require the batchsize to be large enough in general.

\section{Numerical Experiments}
In our numerical experiments, we consider logistic regression with $\ell_2$-regularization -- an instance of \eqref{eq:main_problem} with $f_j(x) = \ln(1+\exp(-y_j \langle a_{j}, x \rangle)) + \frac{\ell_2}{2}\| x\|^2$. Here $\{a_{j}\}_{j \in[m]} \subset \R^d$ are vectors of ``features'', $\{y_j\}_{j\in[m]} \subset \{-1,1\}^m$ are labels, and $\ell_2 \geq 0$ is a  parameter of $\ell_2$-regularization. This problem satisfies Assumption~\ref{as:main} (and also Assumption~\ref{as:UBV} since gradients $\nabla f_j(x)$ are bounded): for each $j\in [m]$ function $f_j$ is $\ell_2$-strongly convex and $L_j$-smooth with $L_j = \ell_2 + \nicefrac{\|a_j\|^2}{4}$ and function $f$ is $L$-smooth with $L = \ell_2 + \nicefrac{\lambda_{\max}(\mA^\top \mA)}{4m}$, where $\mA \in \R^{m\times d}$ is such that the $j$-th row of $\mA$ equals $a_j$ and $\lambda_{\max}(\mA^\top \mA)$ denotes the largest eigenvalue of $\mA^\top \mA$. We chose $l_2 = \nicefrac{L}{1000}$ in all experiments. We consider $4$ datasets from LIBSVM library \cite{chang2011libsvm}: \texttt{a9a} ($m = 32561, d = 123$), \texttt{phishing} ($m = 11055, d = 68$), \texttt{w8a} ($m = 49749, d = 300$) and \texttt{mushrooms} ($m = 8124, d = 112$). The total number of workers in our experiments equals $n = 16$ with $3$ Byzantine workers among them. Byzantine workers use one of the following baseline attacks: $\bullet$ \textbf{Bit Flipping (BF):} Byzantine workers compute $g_i^k$ following the algorithm and send $-g_i^k$ to the server; $\bullet$ \textbf{Label Flipping (LF):} Byzantine workers compute $g_i^k$ with $y_{j_{i,k}^t}$ replaced by $-y_{j_{i,k}^t}$; $\bullet$ \textbf{A Little Is Enough (ALIE) \cite{baruch2019little}:} Byzantine workers compute empirical mean $\mu_{\cG}$ and standard deviation $\sigma_{\cG}$ of $\{g_i^k\}_{i \in \cG}$ and send vector $\mu_{\cG} - z\sigma_{\cG}$ to the server, where $z$ controls the strength of the attack ($z = 1.06$ in our experiments); $\bullet$ \textbf{Inner Product Manipulation (IPM) \cite{xie2020fall}:} Byzantine workers send $-\frac{\varepsilon}{G}\sum_{i\in \cG} g_i^k$ to the server, where $\varepsilon > 0$ is a parameter ($\varepsilon = 0.1$ in our experiments). Our code is publicly available: \url{https://github.com/Nikosimus/BR-LSVRG}.
\newline

\noindent\textbf{Experiment 1:} \algname{BR-LSVRG} \textbf{with different batchsizes.} In this experiment, we tested \algname{BR-LSVRG} with two different batchsizes: 1 and $0.01m$. Stepsize was chosen as $\gamma = \nicefrac{1}{12 L}$. In all runs, \algname{BR-LSVRG} with moderate batchsize $b = 0.01m$ achieves a very high accuracy of the solution. In addition, \algname{BR-LSVRG} with batchsize $b = 1$ always achieves at least $10^{-5}$ functional suboptimality, which is a relatively good accuracy as well. This experiment illustrates that \algname{BR-LSVRG} can converge to high accuracy even with small or moderate batchsizes.
\newline

\begin{figure}[t]
\begin{minipage}[h]{0.24\linewidth}
\center{\includegraphics[width=1\linewidth]{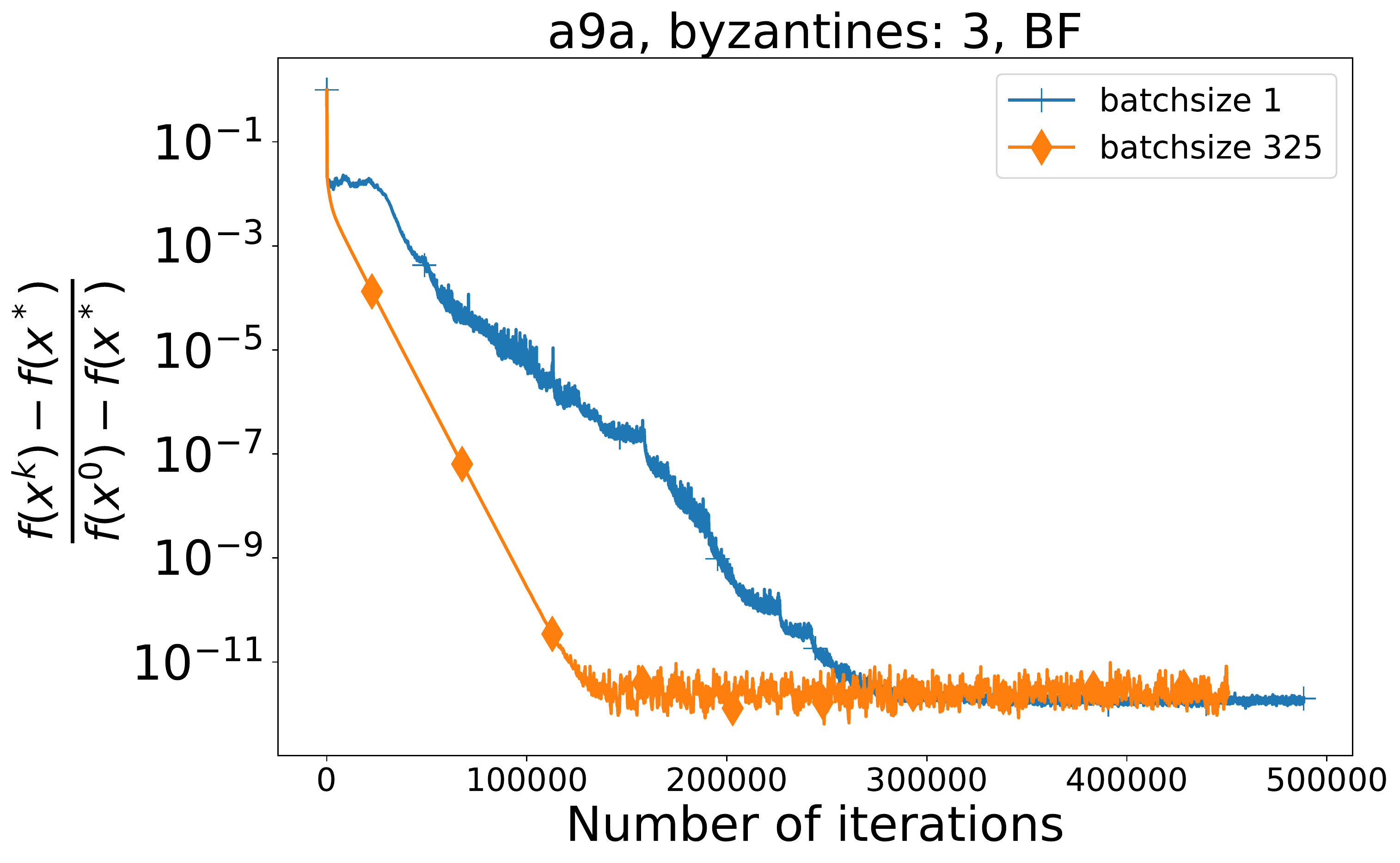}} \\
\end{minipage}
\hfill
\begin{minipage}[h]{0.24\linewidth}
\center{\includegraphics[width=1\linewidth]{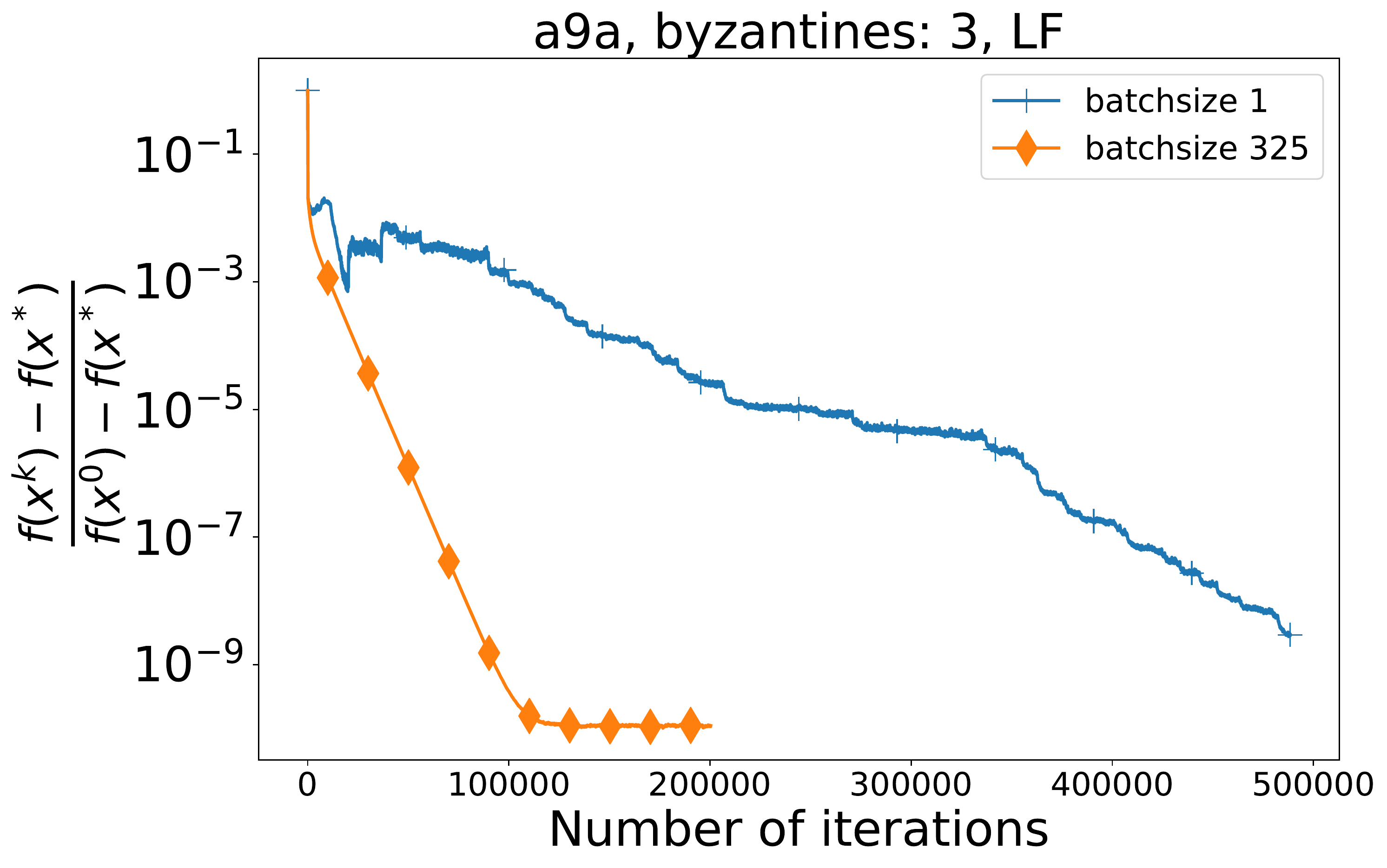}} \\
\end{minipage}
\hfill
\begin{minipage}[h]{0.24\linewidth}
\center{\includegraphics[width=1\linewidth]{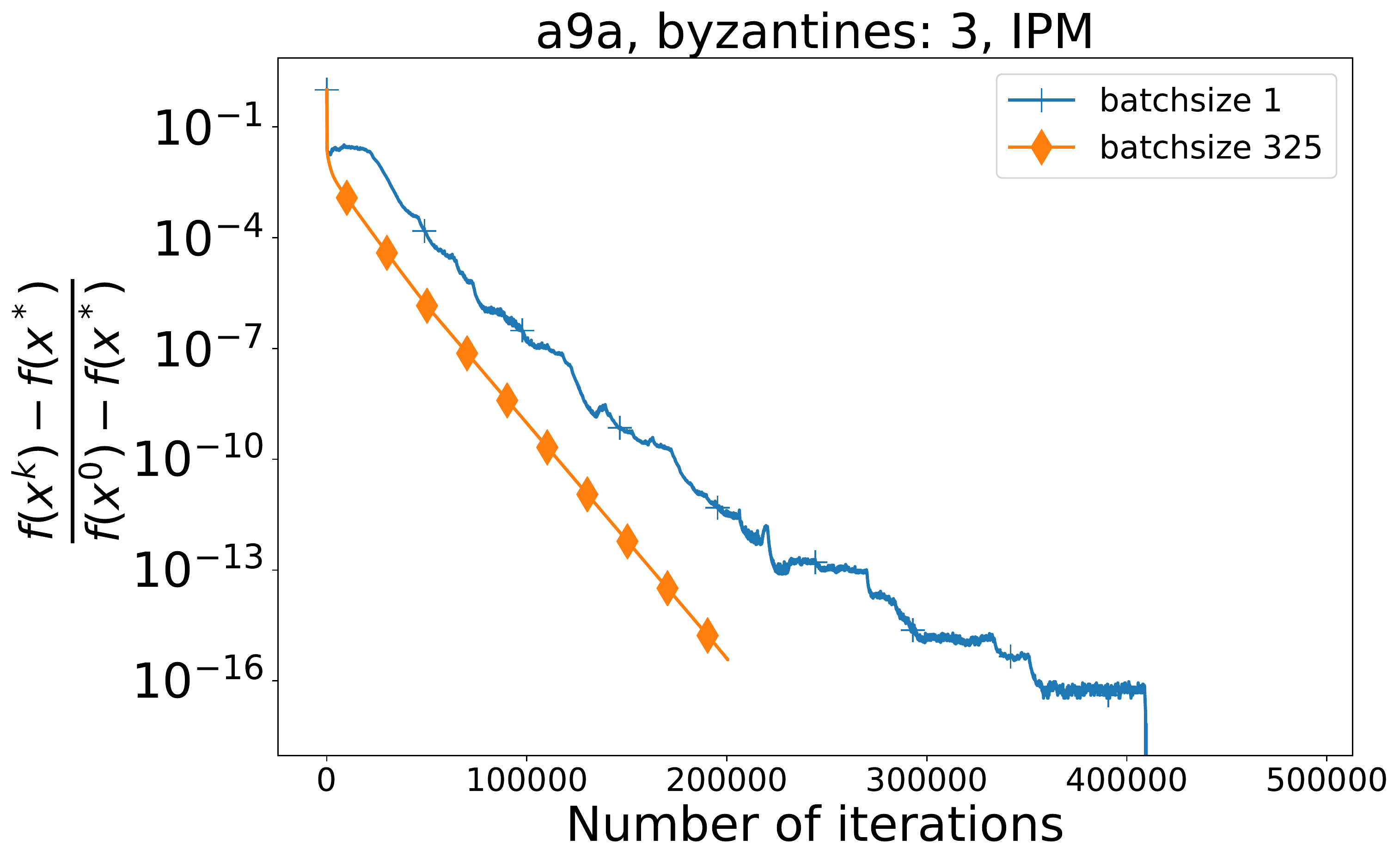}} \\
\end{minipage}
\hfill
\begin{minipage}[h]{0.24\linewidth}
\center{\includegraphics[width=1\linewidth]{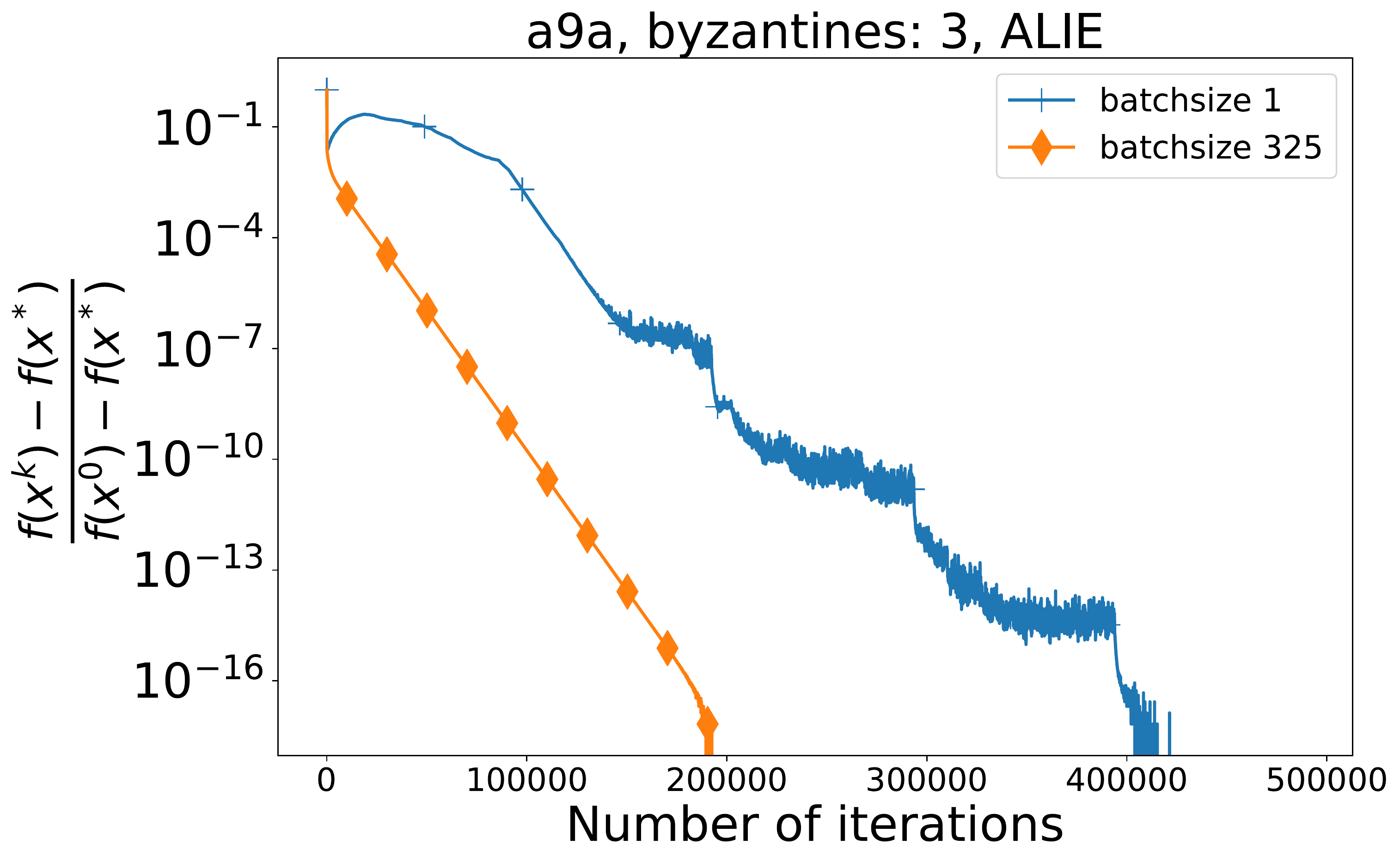}} \\
\end{minipage}
\vfill
\begin{minipage}[h]{0.24\linewidth}
\center{\includegraphics[width=1\linewidth]{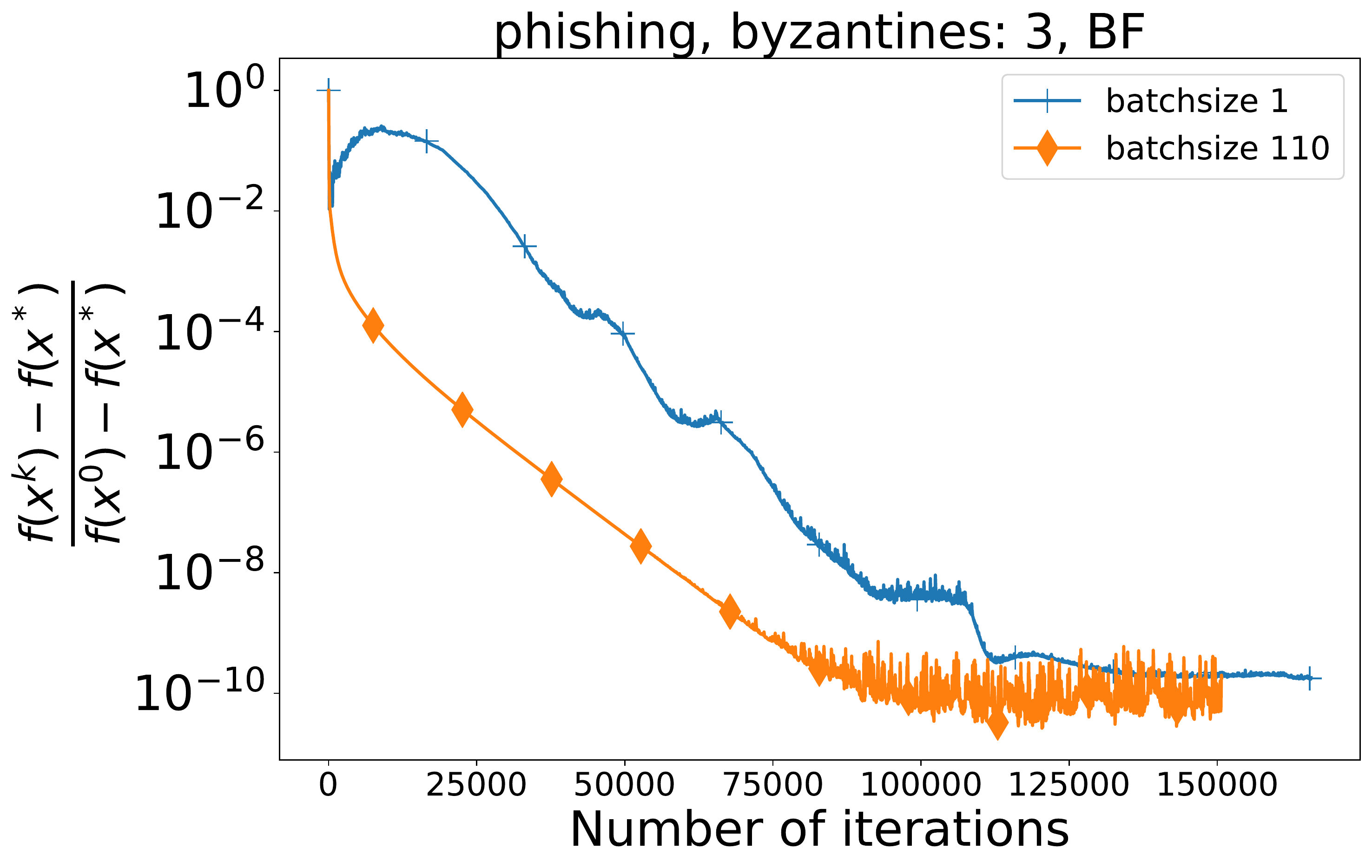}} \\
\end{minipage}
\hfill
\begin{minipage}[h]{0.24\linewidth}
\center{\includegraphics[width=1\linewidth]{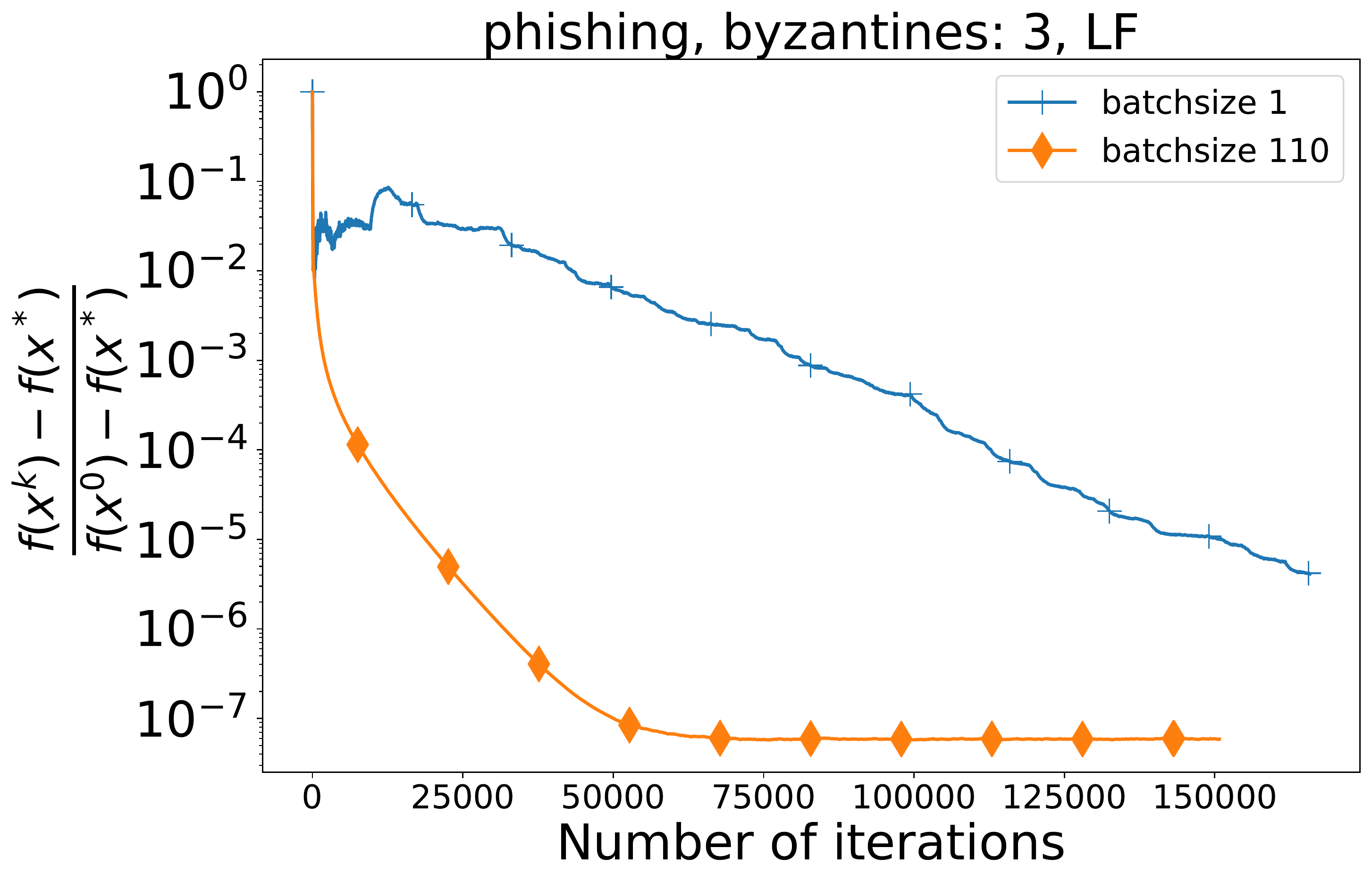}} \\
\end{minipage}
\hfill
\begin{minipage}[h]{0.24\linewidth}
\center{\includegraphics[width=1\linewidth]{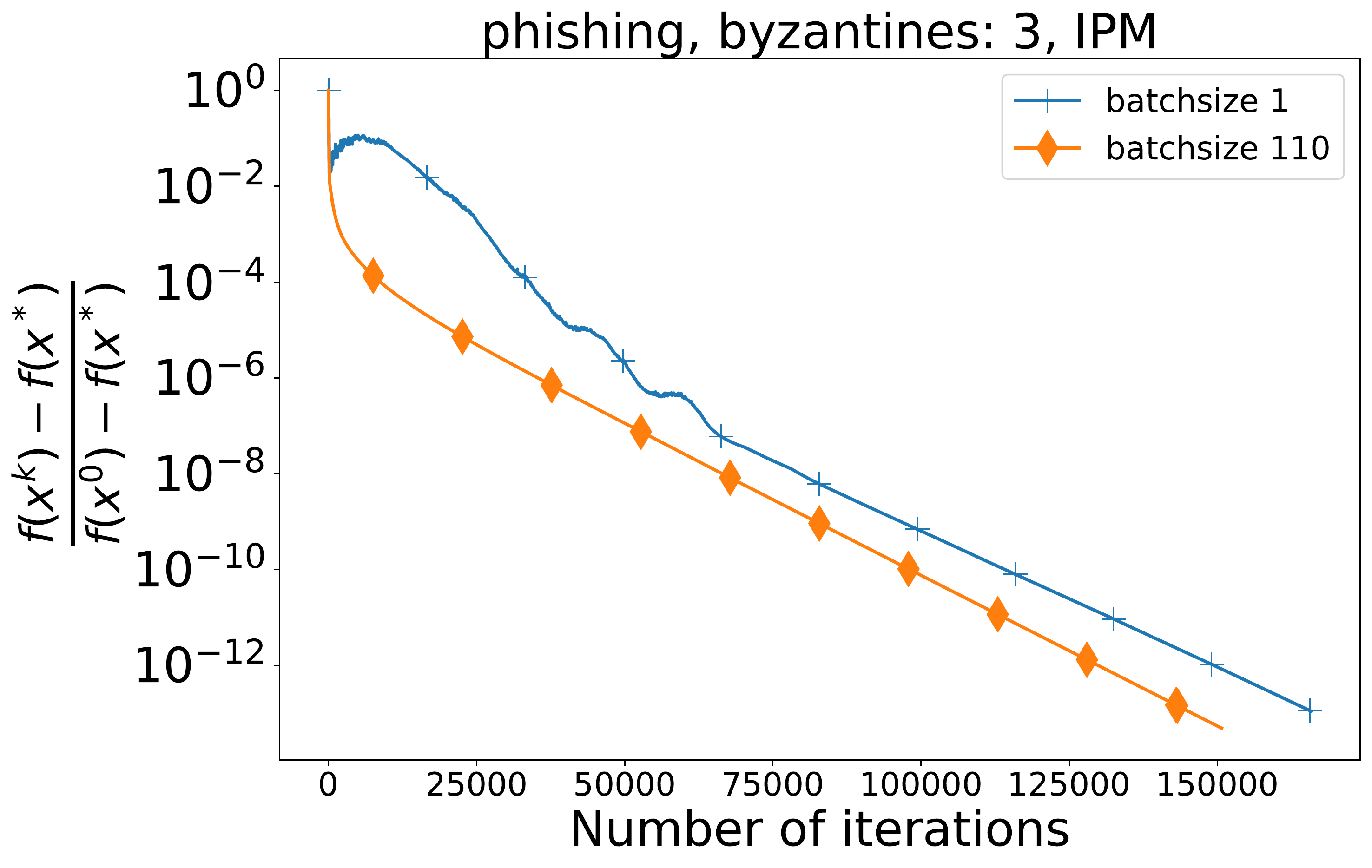}} \\
\end{minipage}
\hfill
\begin{minipage}[h]{0.24\linewidth}
\center{\includegraphics[width=1\linewidth]{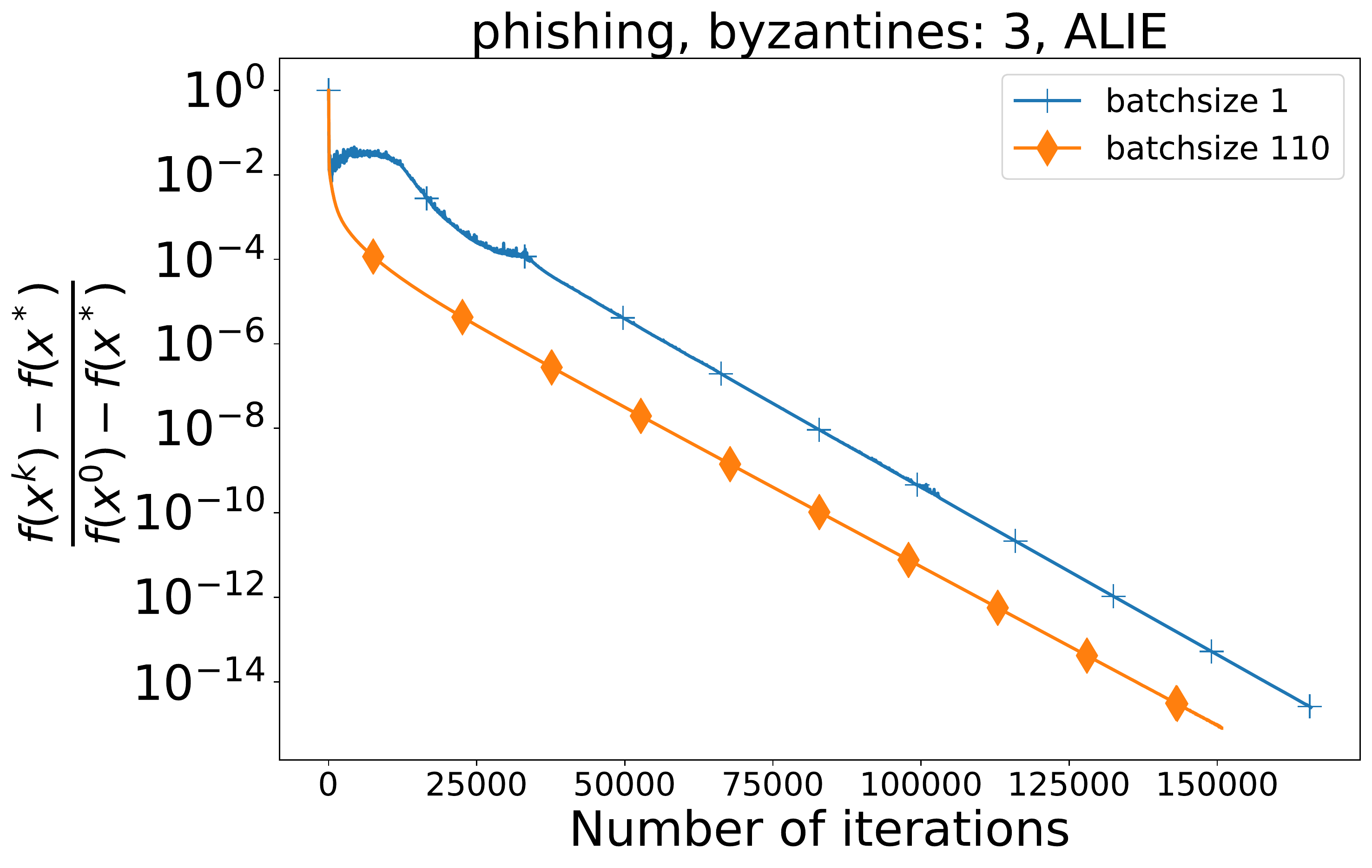}} \\
\end{minipage}
\vfill
\begin{minipage}[h]{0.24\linewidth}
\center{\includegraphics[width=1\linewidth]{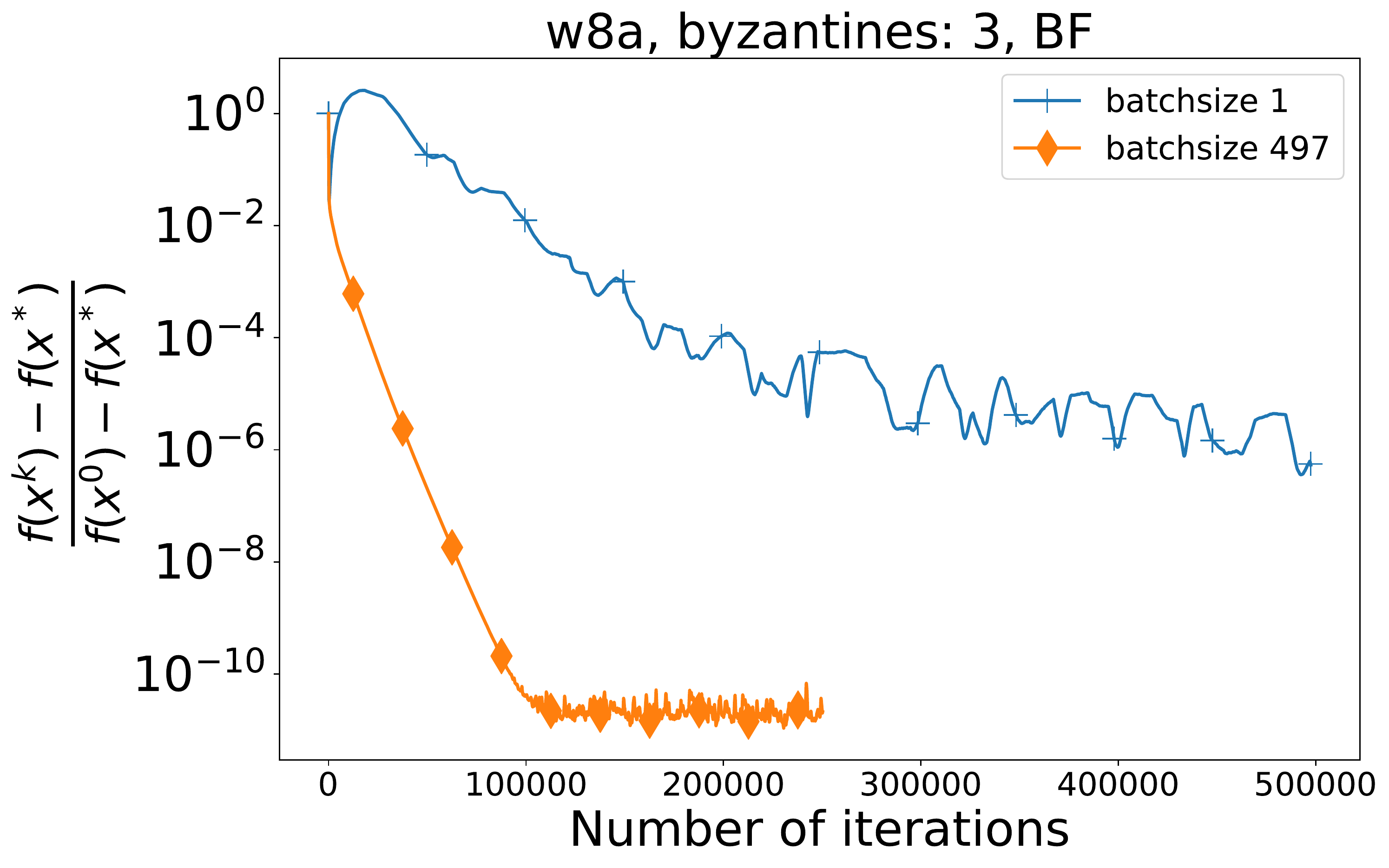}} \\
\end{minipage}
\hfill
\begin{minipage}[h]{0.24\linewidth}
\center{\includegraphics[width=1\linewidth]{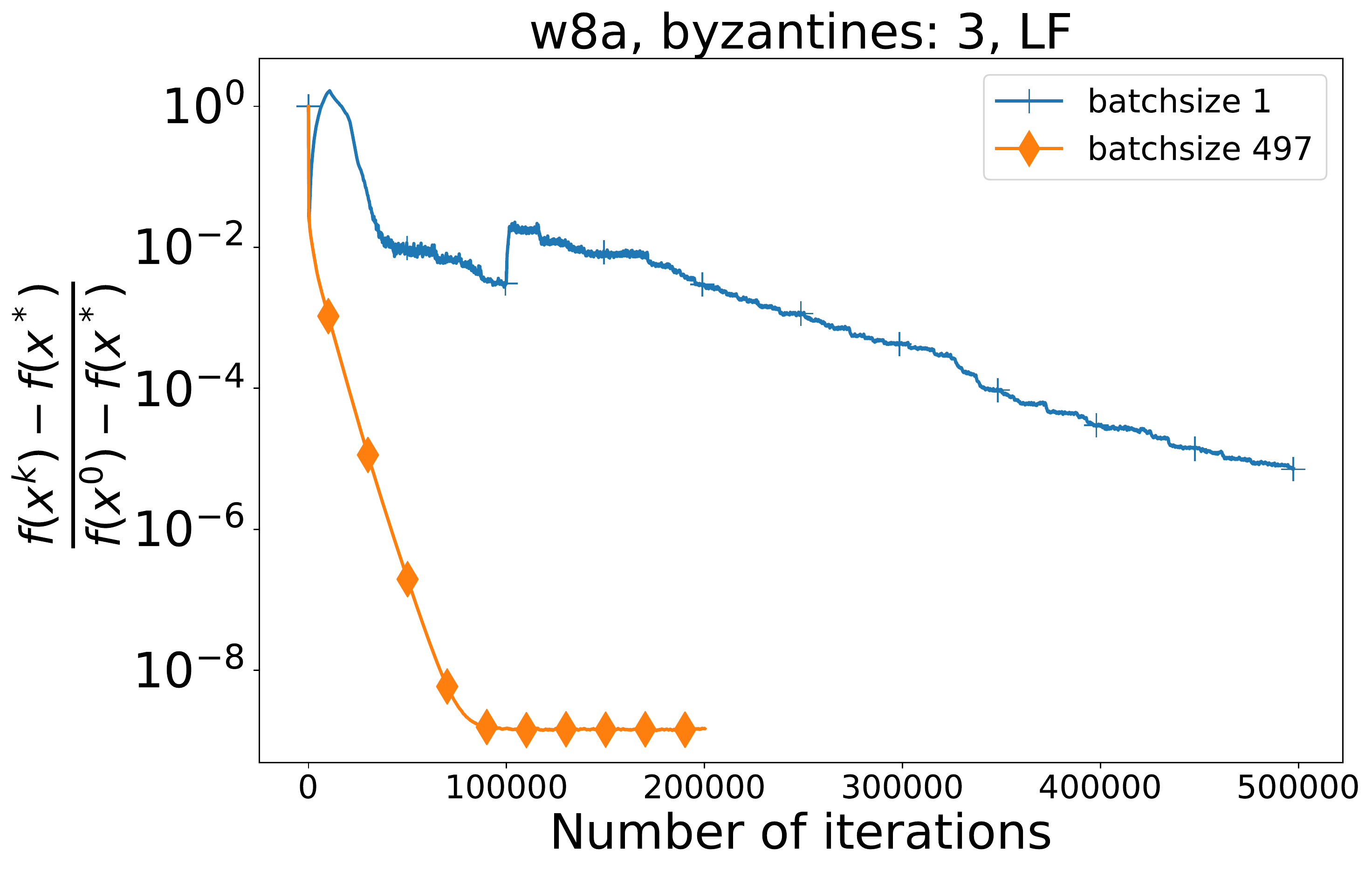}} \\
\end{minipage}
\hfill
\begin{minipage}[h]{0.24\linewidth}
\center{\includegraphics[width=1\linewidth]{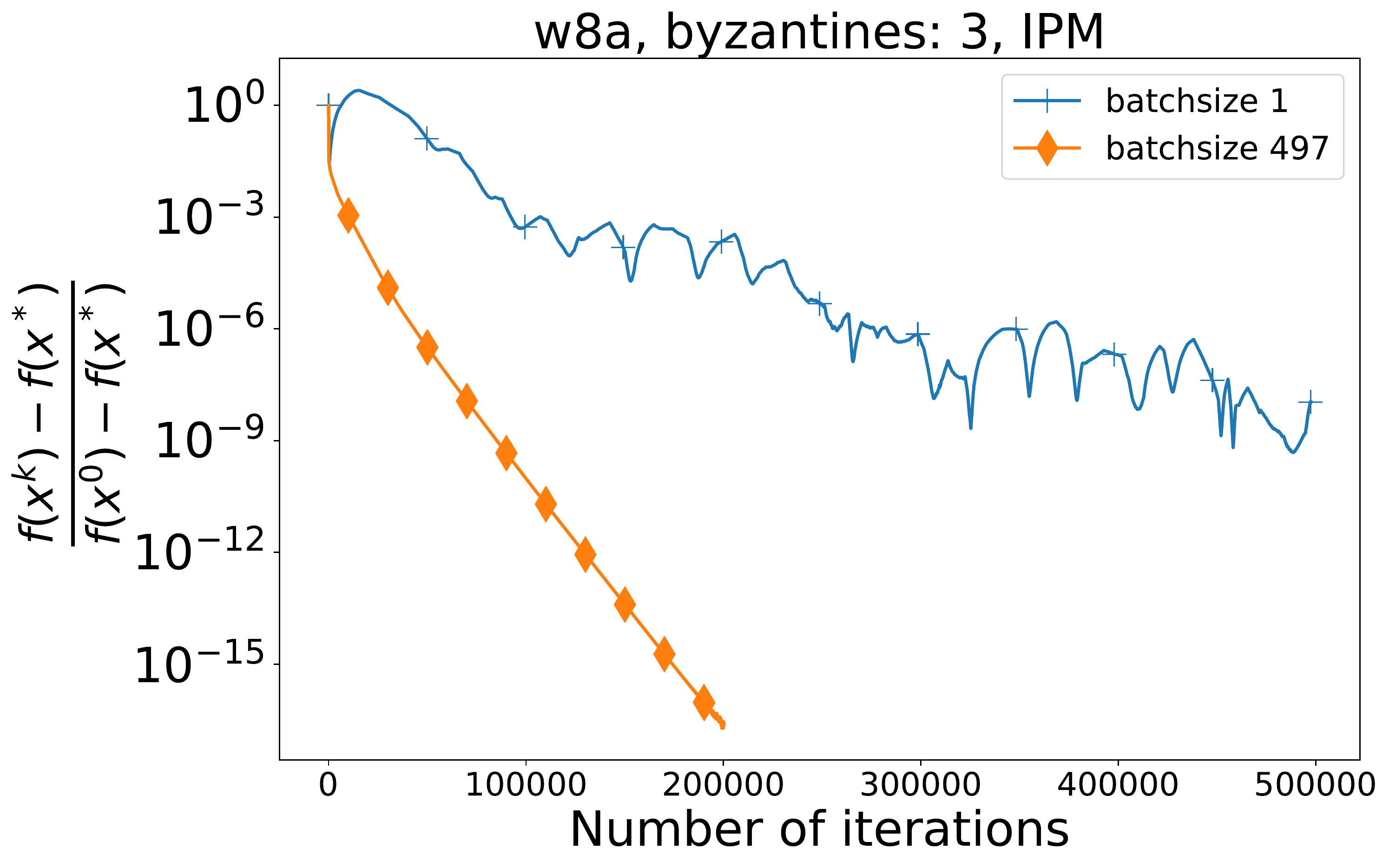}} \\
\end{minipage}
\hfill
\begin{minipage}[h]{0.24\linewidth}
\center{\includegraphics[width=1\linewidth]{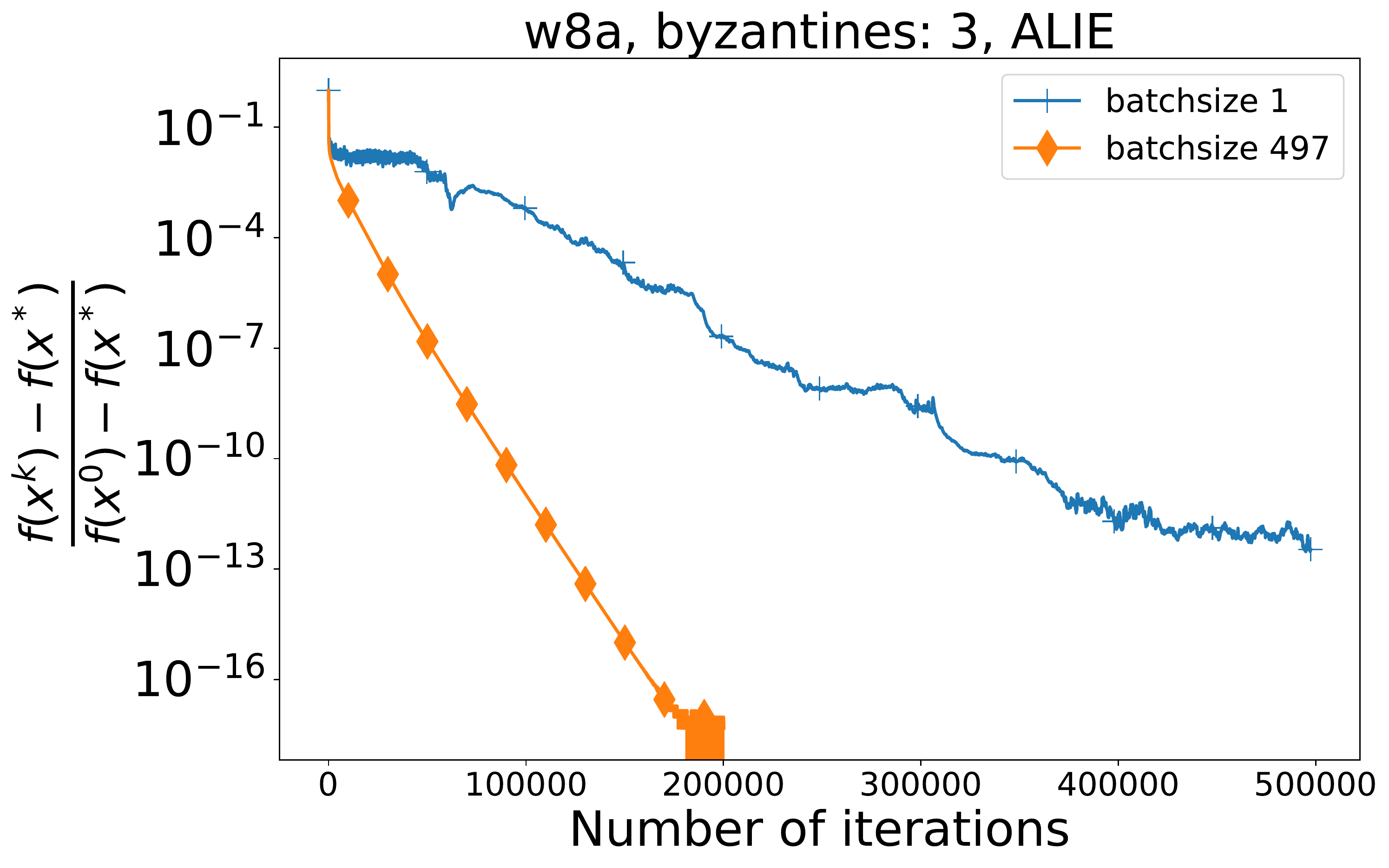}} \\
\end{minipage}
\vfill
\begin{minipage}[h]{0.24\linewidth}
\center{\includegraphics[width=1\linewidth]{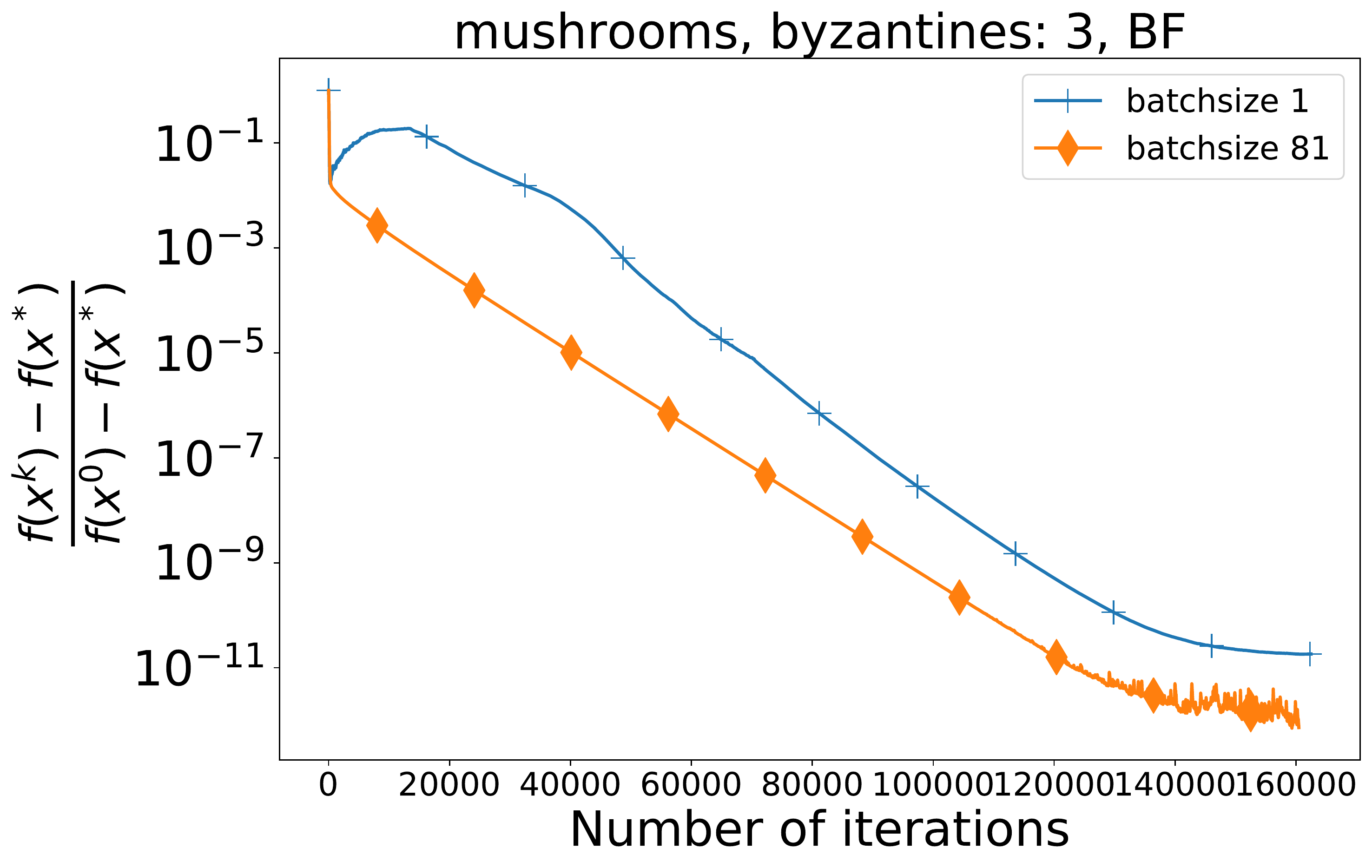}} \\
\end{minipage}
\hfill
\begin{minipage}[h]{0.24\linewidth}
\center{\includegraphics[width=1\linewidth]{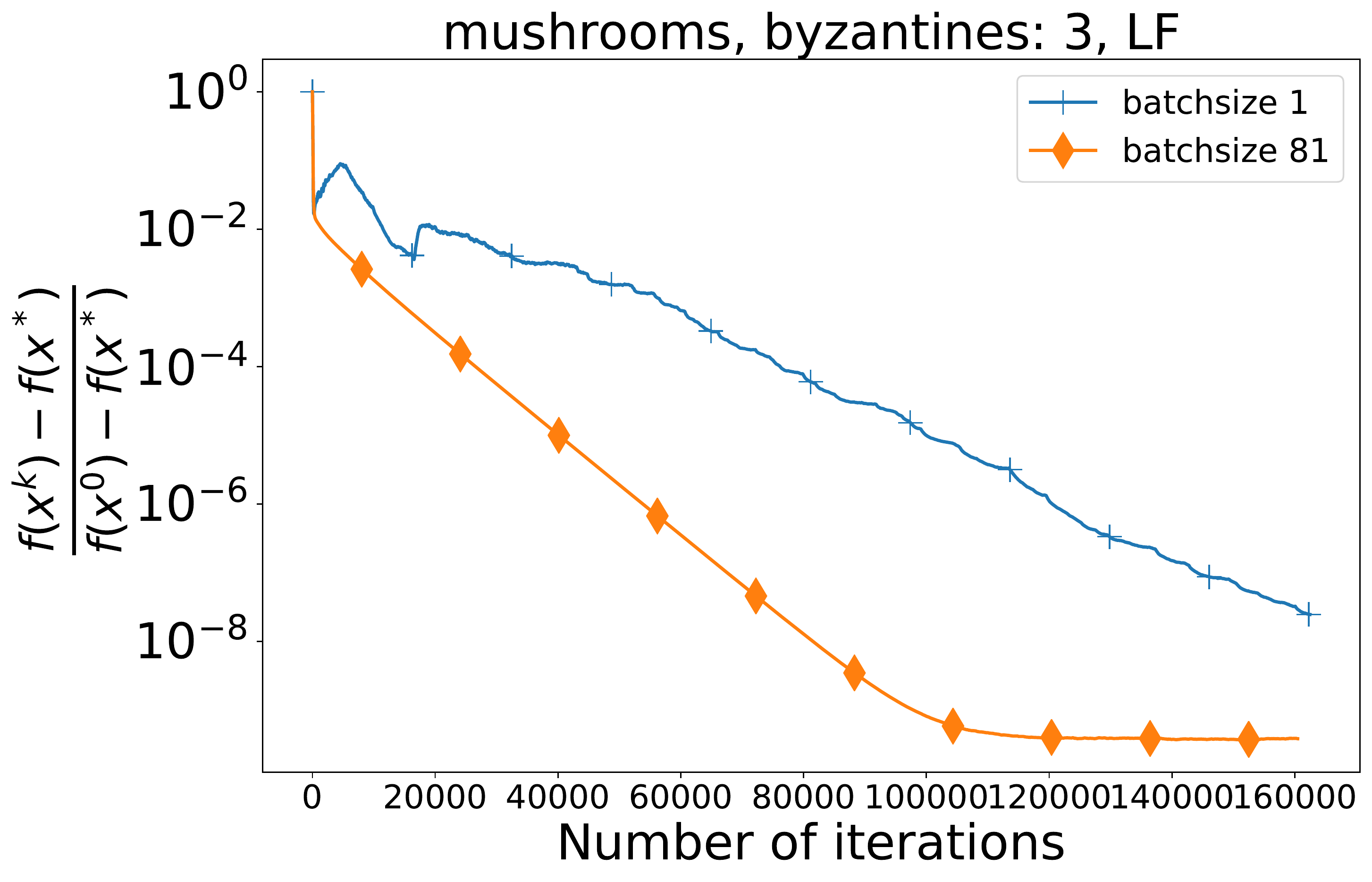}} \\
\end{minipage}
\hfill
\begin{minipage}[h]{0.24\linewidth}
\center{\includegraphics[width=1\linewidth]{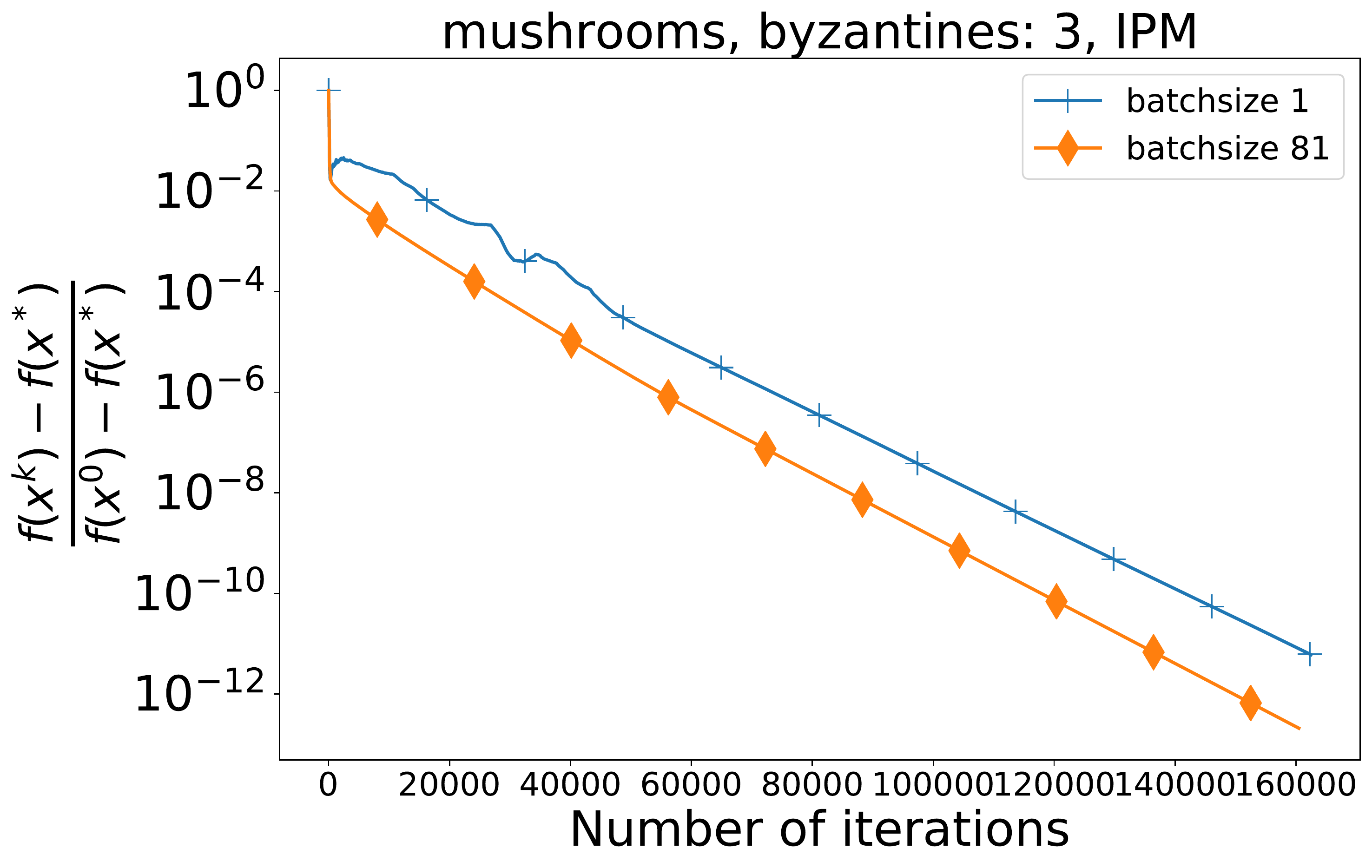}} \\
\end{minipage}
\hfill
\begin{minipage}[h]{0.24\linewidth}
\center{\includegraphics[width=1\linewidth]{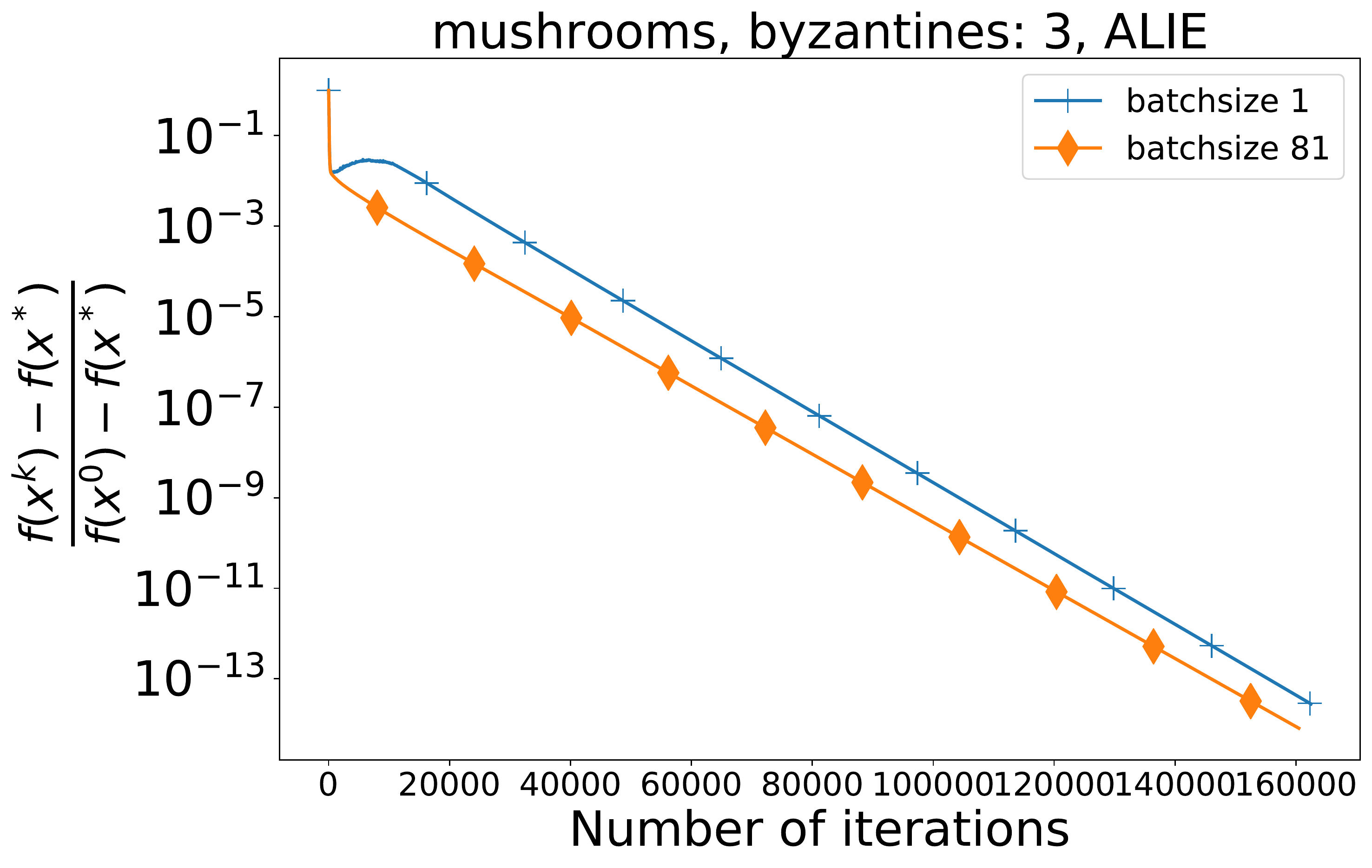}} \\
\end{minipage}
\vfill
\caption{Trajectories of \algname {BR-LSVRG} with batchsizes 1 and $0.01m$. Each row corresponds to one of the used datasets with 4 different types of attacks.}
\label{first_exp}
\end{figure}

\noindent\textbf{Experiment 2: comparison with \algname{Byz-VR-MARINA} and \algname{Byrd-SAGA}.} Next, we compare \algname{BR-LSVRG} with \algname{Byz-VR-MARINA} and \algname{Byrd-SAGA}. All the methods were run with stepsize $\gamma = \nicefrac{5}{2L}$ and batchsize $0.01m$. In all cases, \algname{BR-LSVRG} achieves better accuracy than \algname{Byrd-SAGA} and shows a comparable convergence to \algname{Byz-VR-MARINA} to the very high accuracy. We also tested \algname{Byrd-SAGA} with smaller stepsizes, but the method did not achieve better accuracy.

\begin{figure}[t]
\begin{minipage}[h]{0.24\linewidth}
\center{\includegraphics[width=1\linewidth]{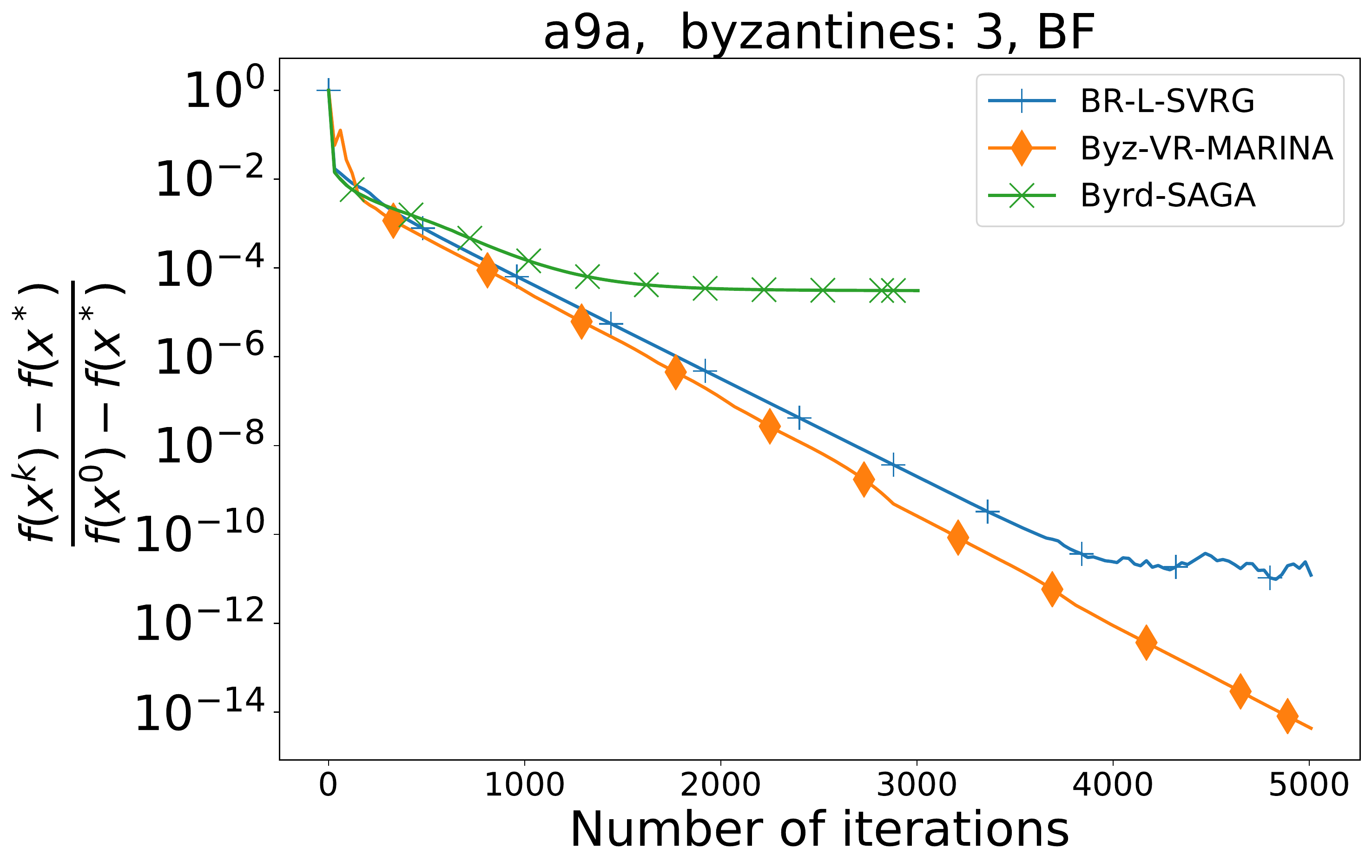}} \\
\end{minipage}
\hfill
\begin{minipage}[h]{0.24\linewidth}
\center{\includegraphics[width=1\linewidth]{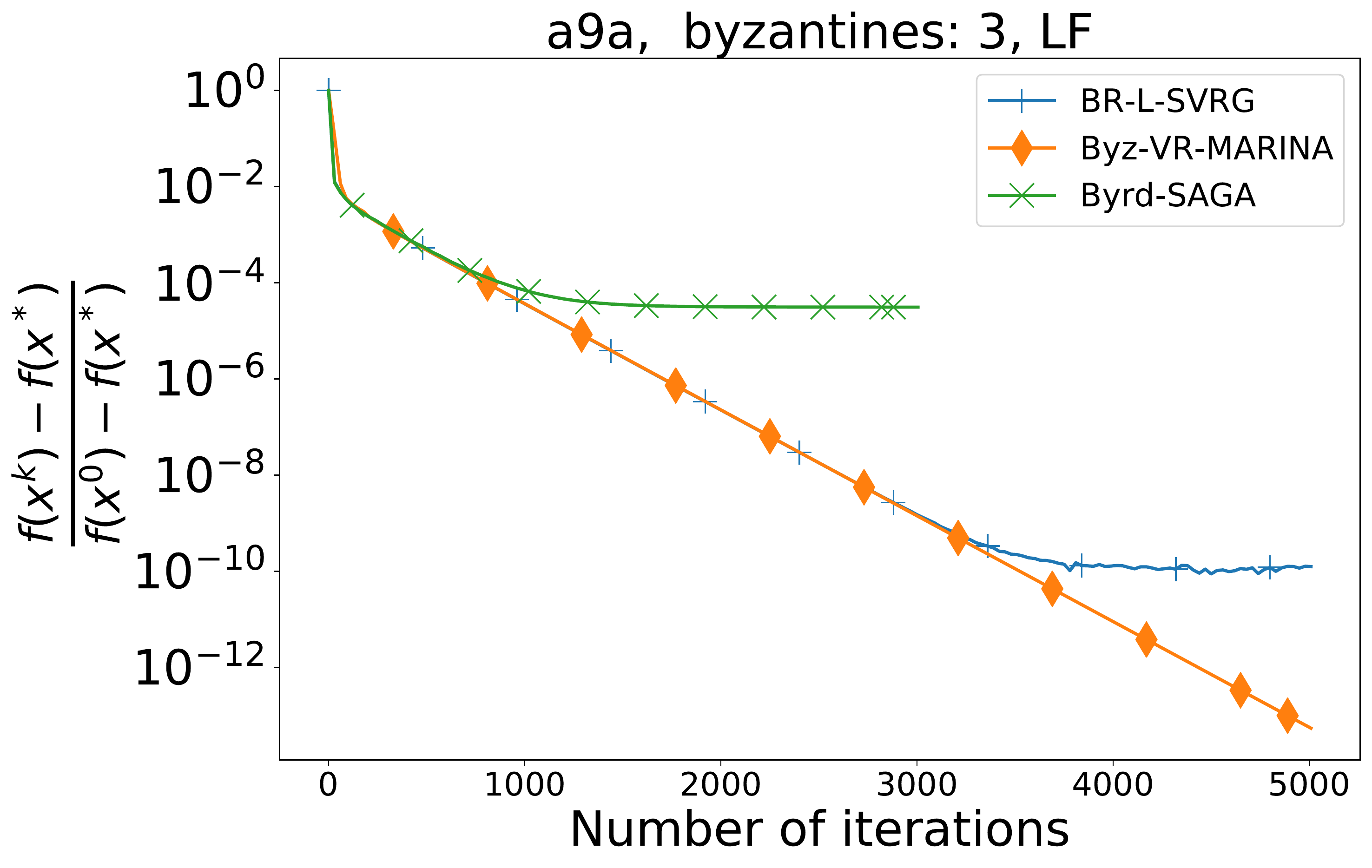}} \\
\end{minipage}
\hfill
\begin{minipage}[h]{0.24\linewidth}
\center{\includegraphics[width=1\linewidth]{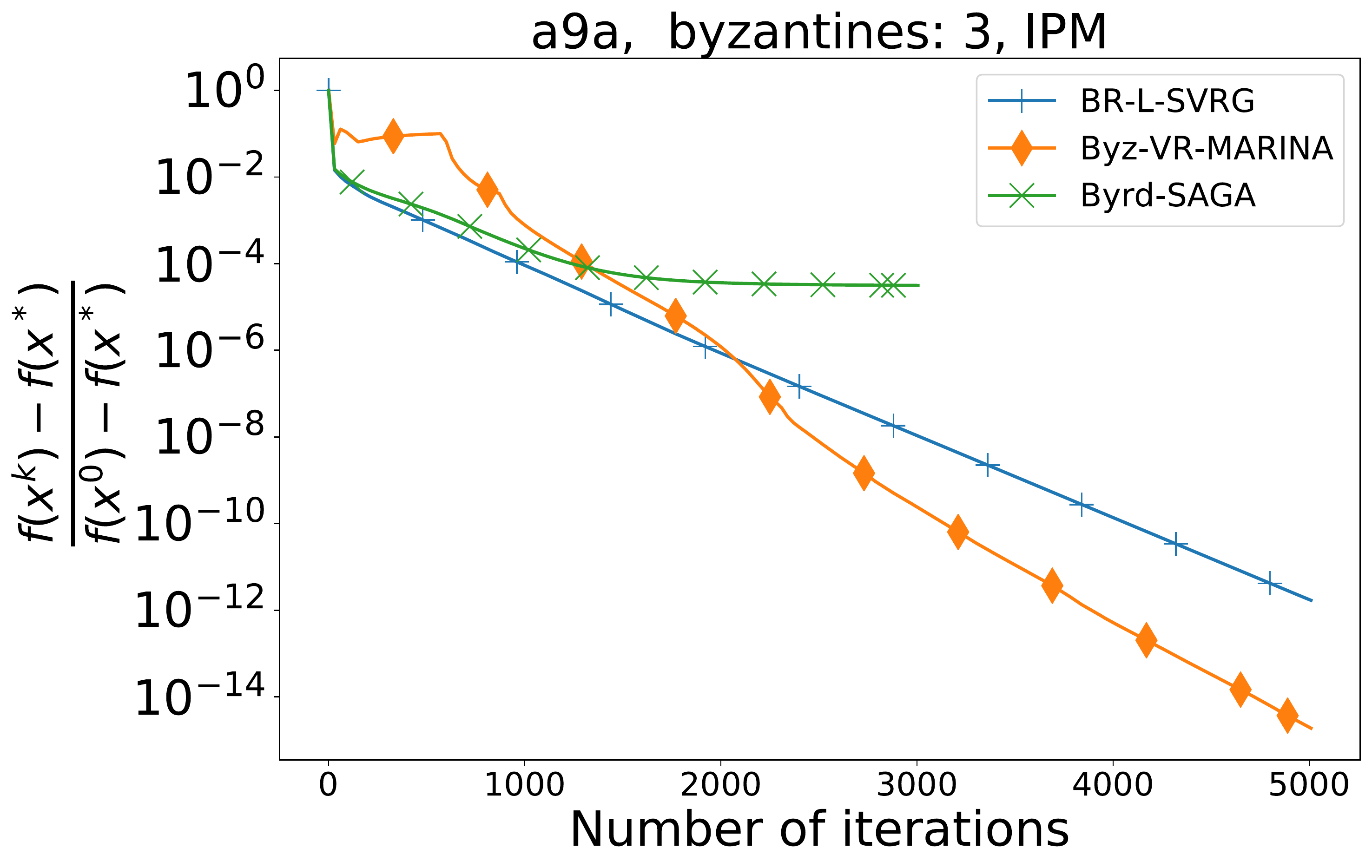}} \\
\end{minipage}
\hfill
\begin{minipage}[h]{0.24\linewidth}
\center{\includegraphics[width=1\linewidth]{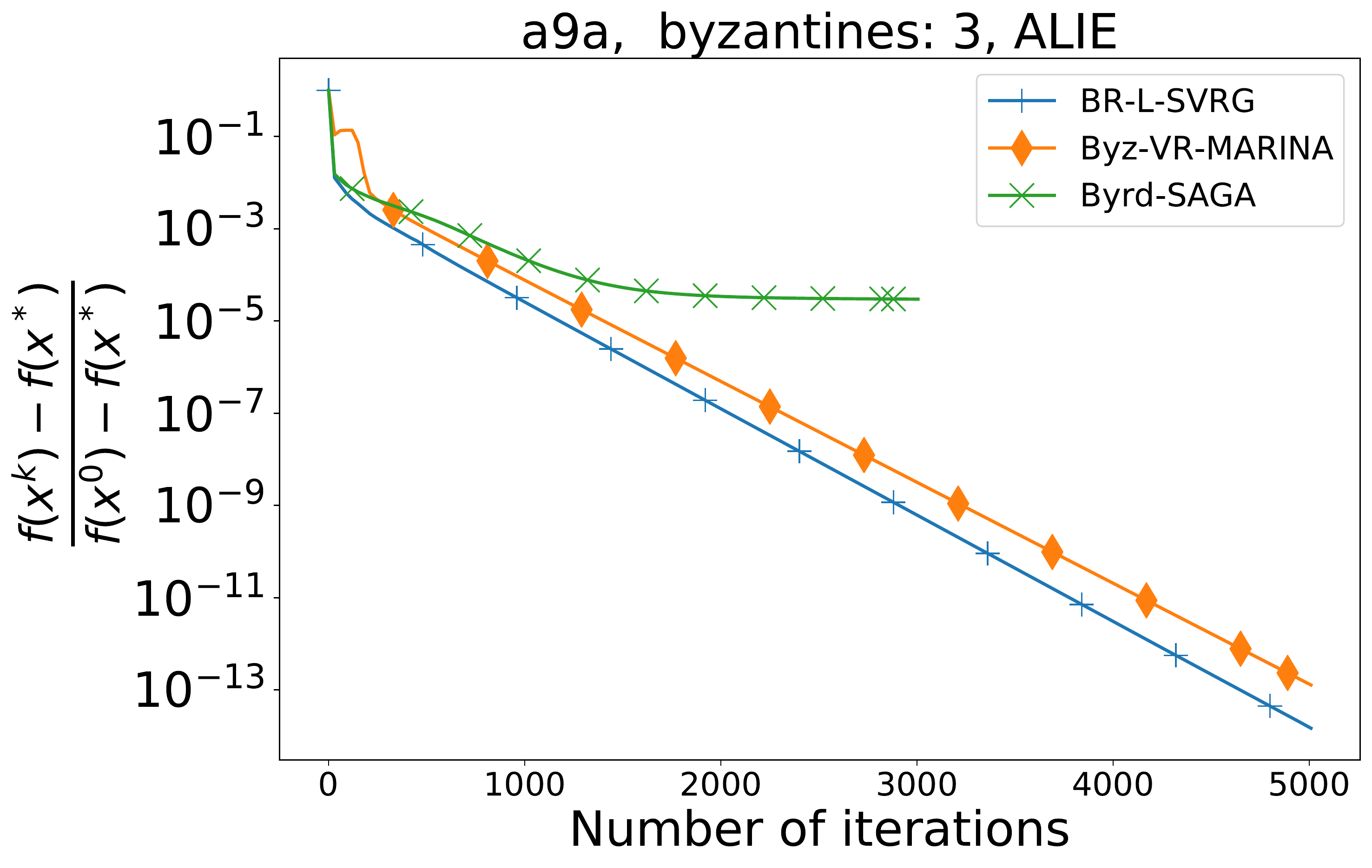}} \\
\end{minipage}
\vfill
\begin{minipage}[h]{0.24\linewidth}
\center{\includegraphics[width=1\linewidth]{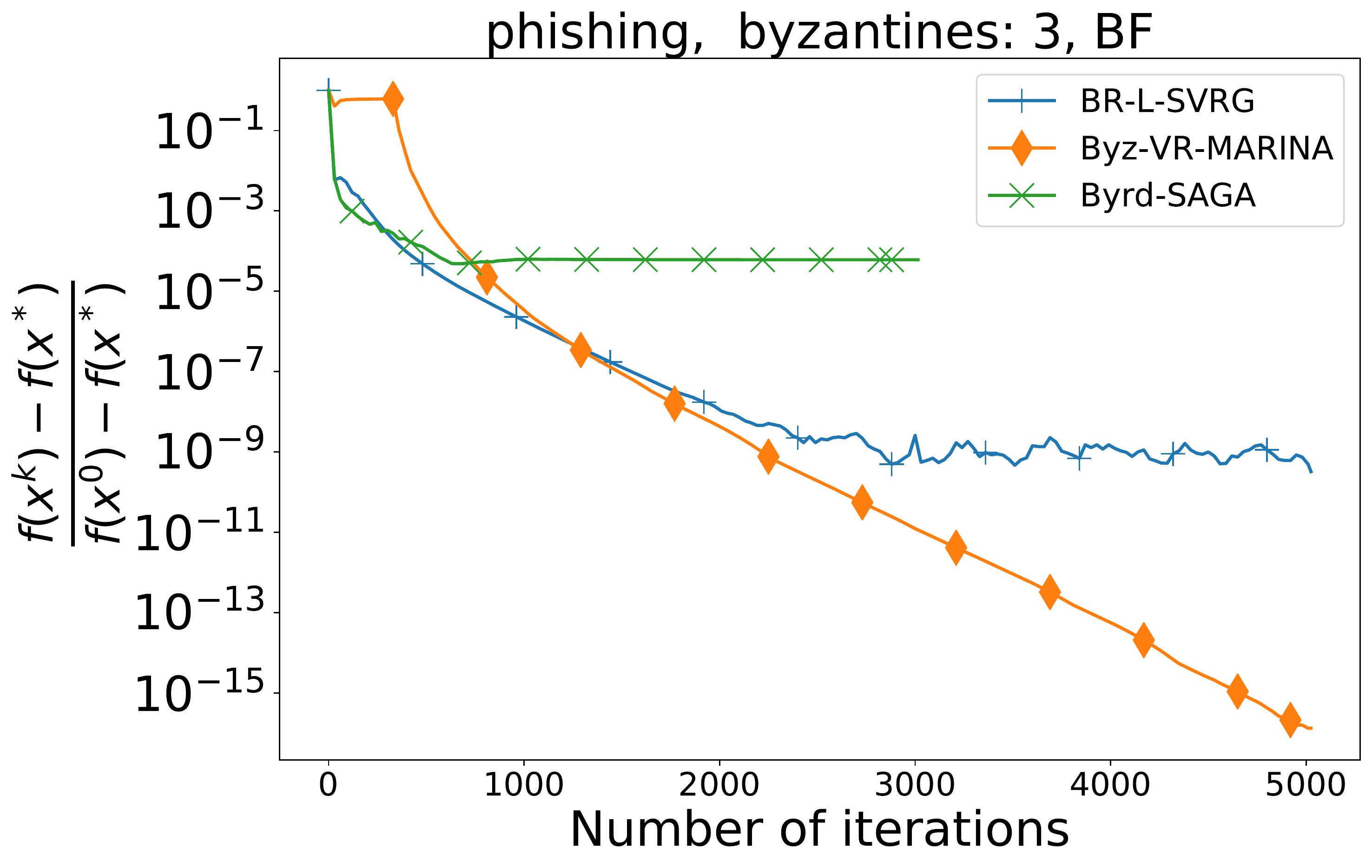}} \\
\end{minipage}
\hfill
\begin{minipage}[h]{0.24\linewidth}
\center{\includegraphics[width=1\linewidth]{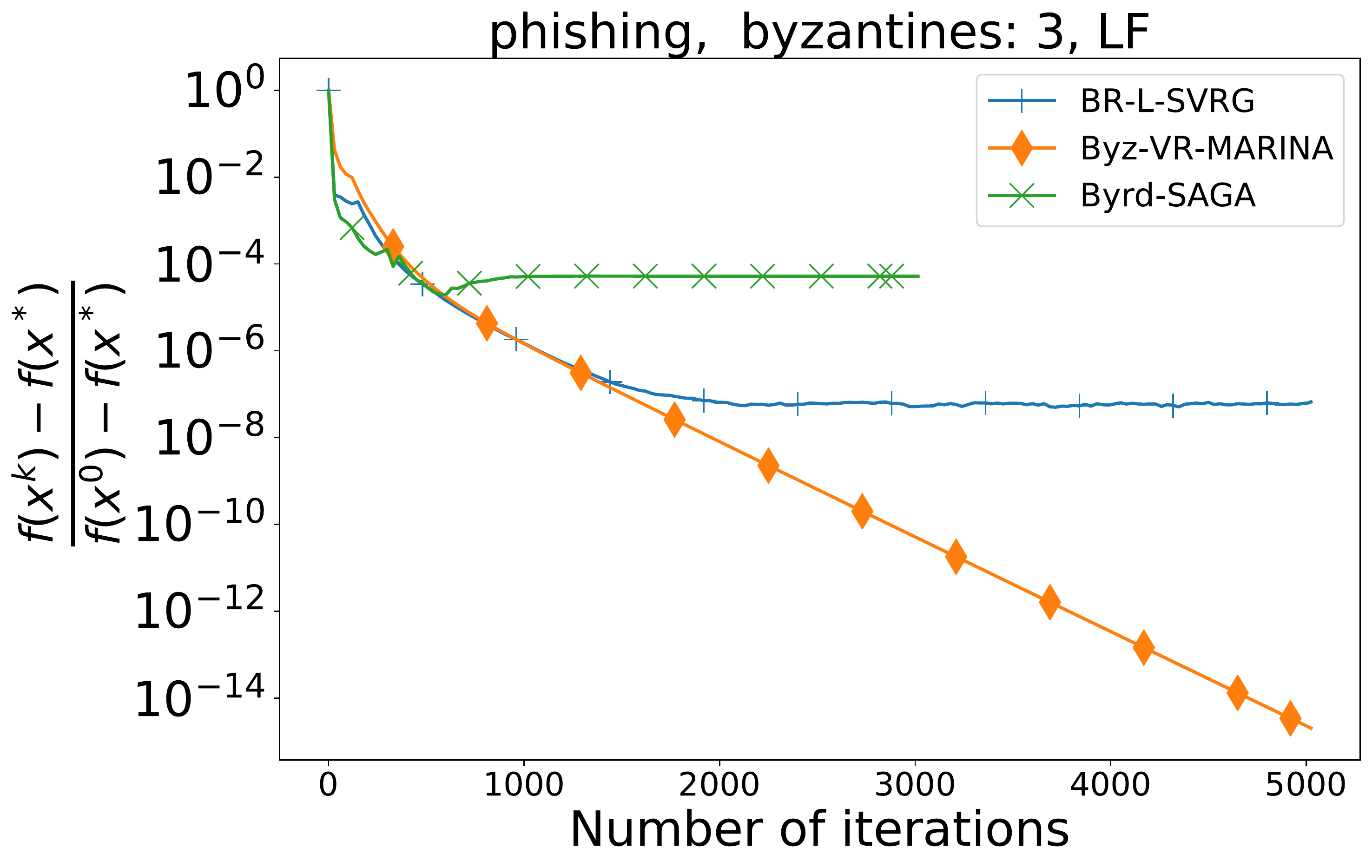}} \\
\end{minipage}
\hfill
\begin{minipage}[h]{0.24\linewidth}
\center{\includegraphics[width=1\linewidth]{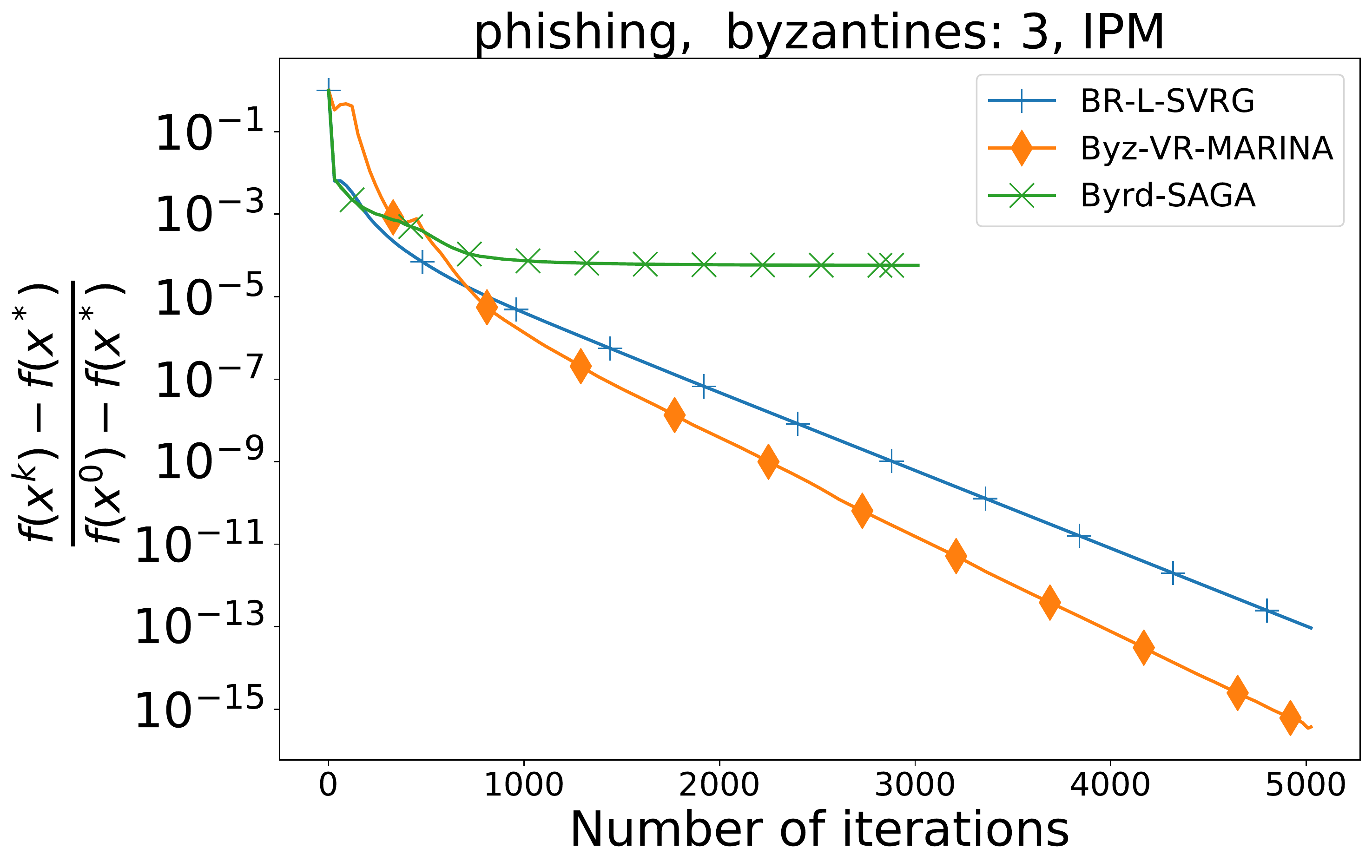}} \\
\end{minipage}
\hfill
\begin{minipage}[h]{0.24\linewidth}
\center{\includegraphics[width=1\linewidth]{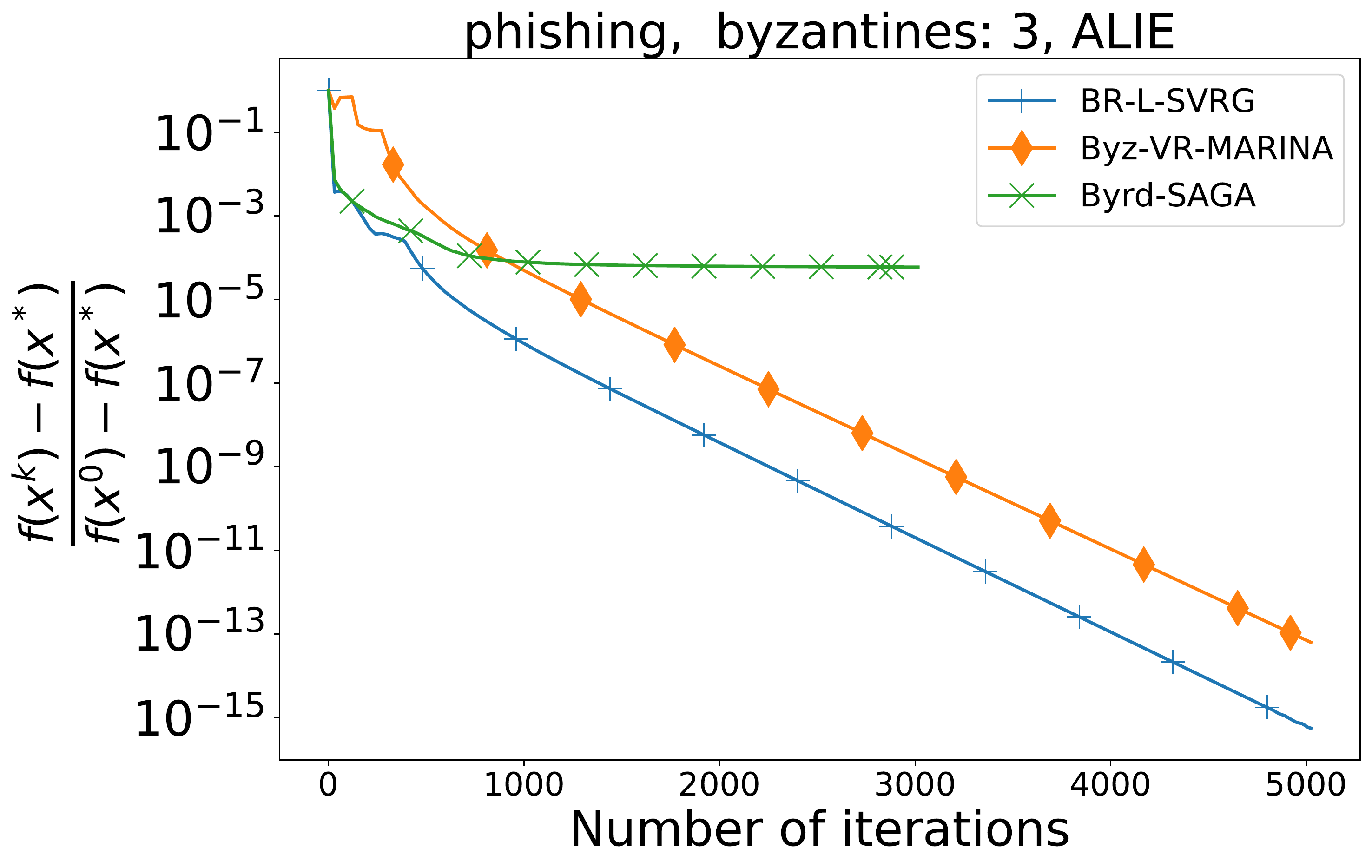}} \\
\end{minipage}
\vfill
\begin{minipage}[h]{0.24\linewidth}
\center{\includegraphics[width=1\linewidth]{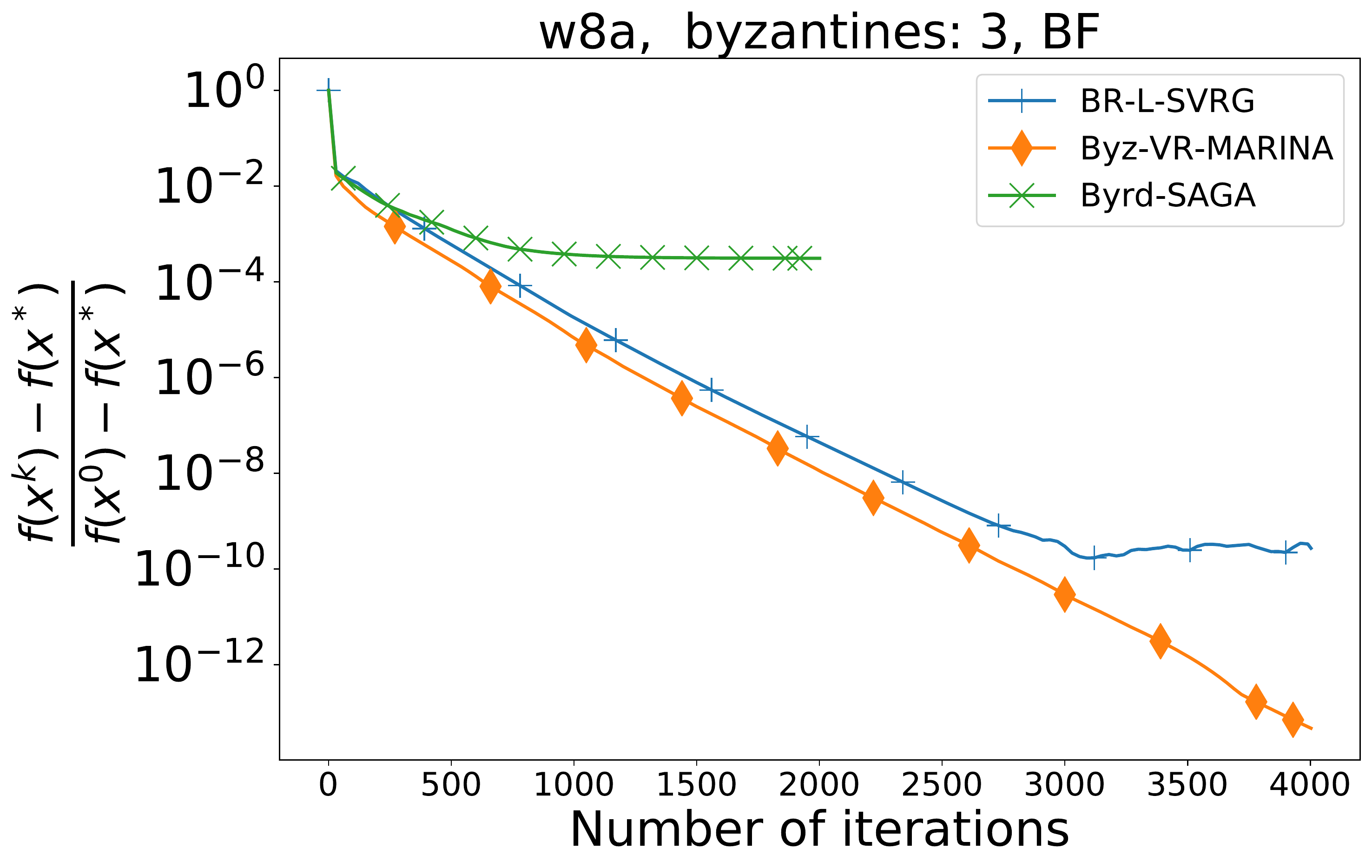}} \\
\end{minipage}
\hfill
\begin{minipage}[h]{0.24\linewidth}
\center{\includegraphics[width=1\linewidth]{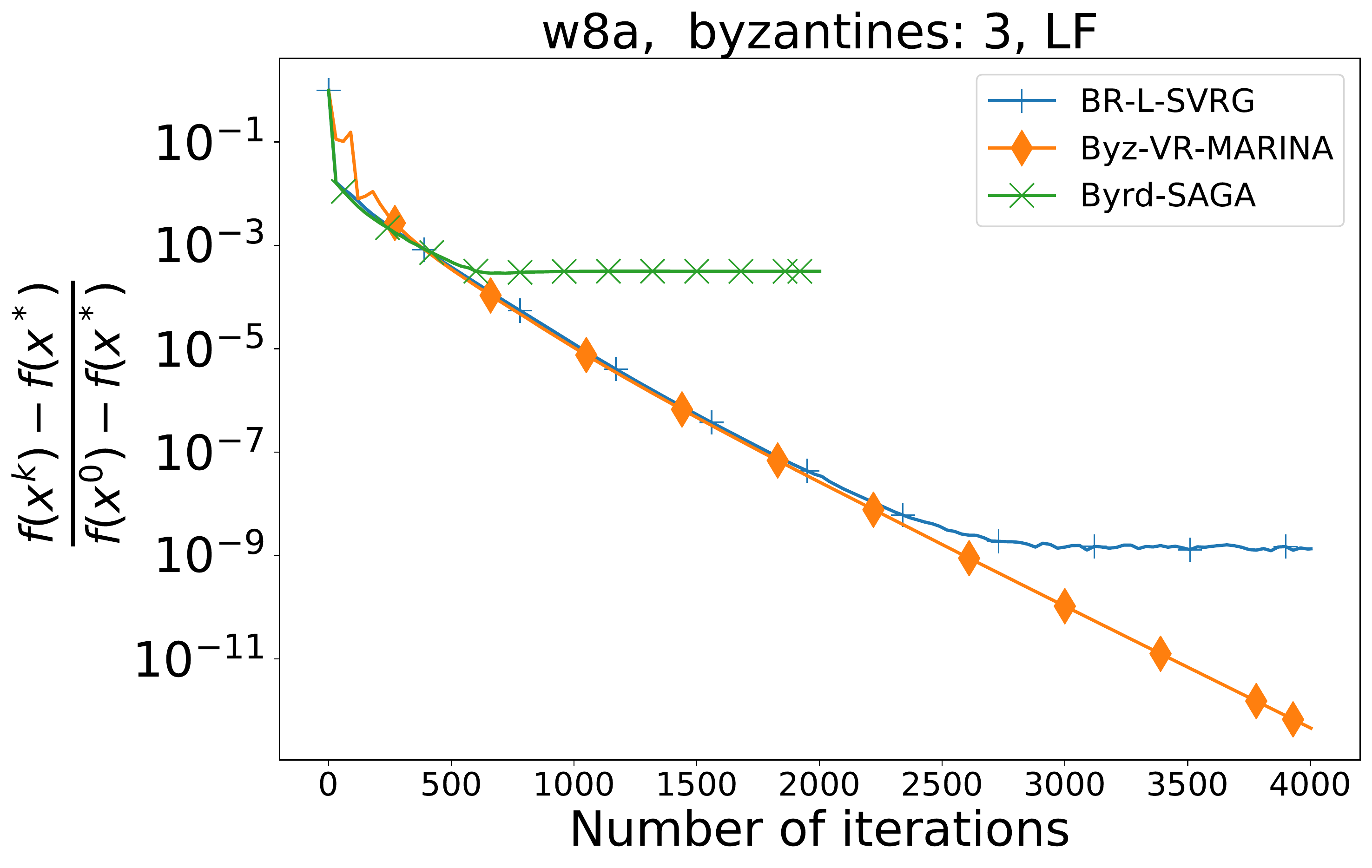}} \\
\end{minipage}
\hfill
\begin{minipage}[h]{0.24\linewidth}
\center{\includegraphics[width=1\linewidth]{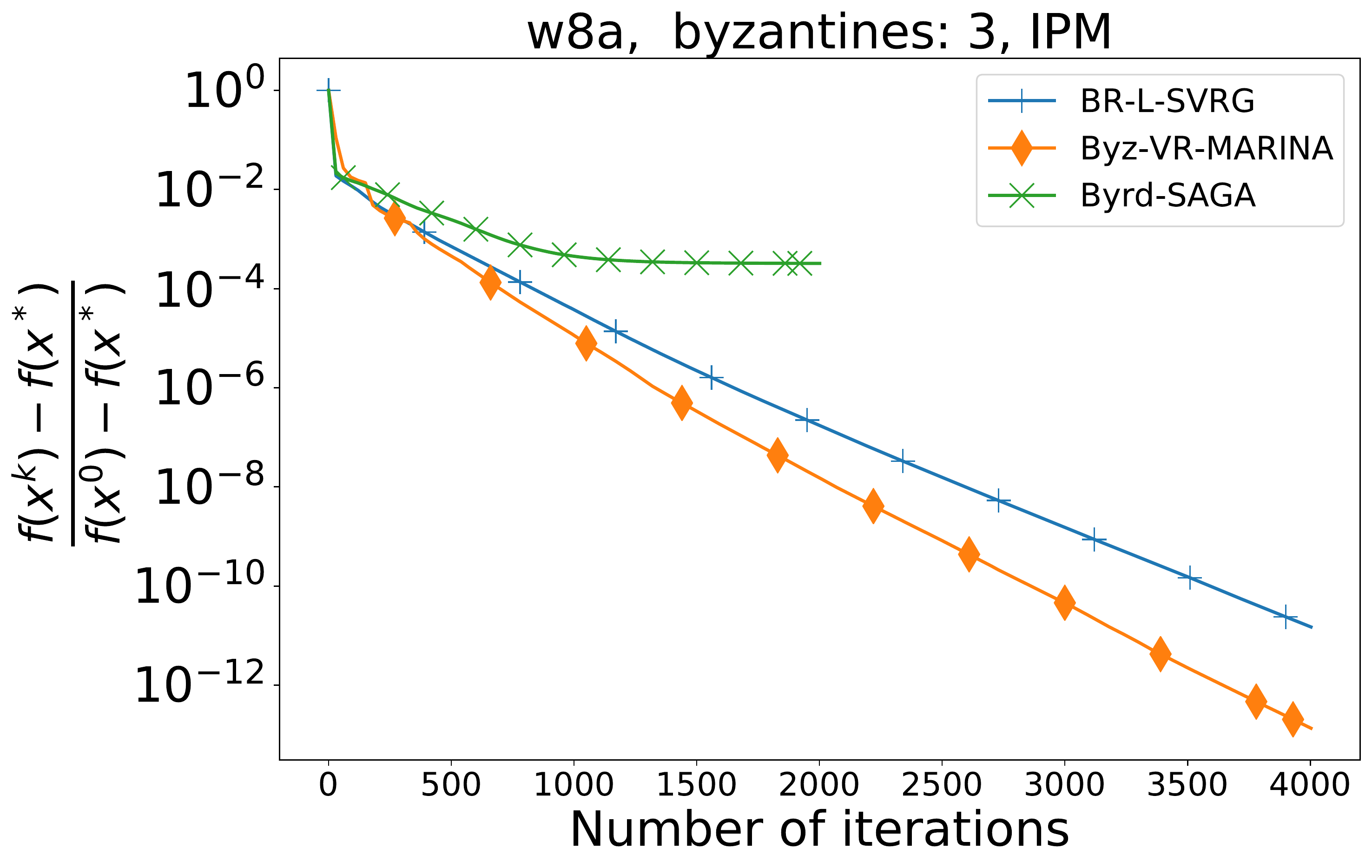}} \\
\end{minipage}
\hfill
\begin{minipage}[h]{0.24\linewidth}
\center{\includegraphics[width=1\linewidth]{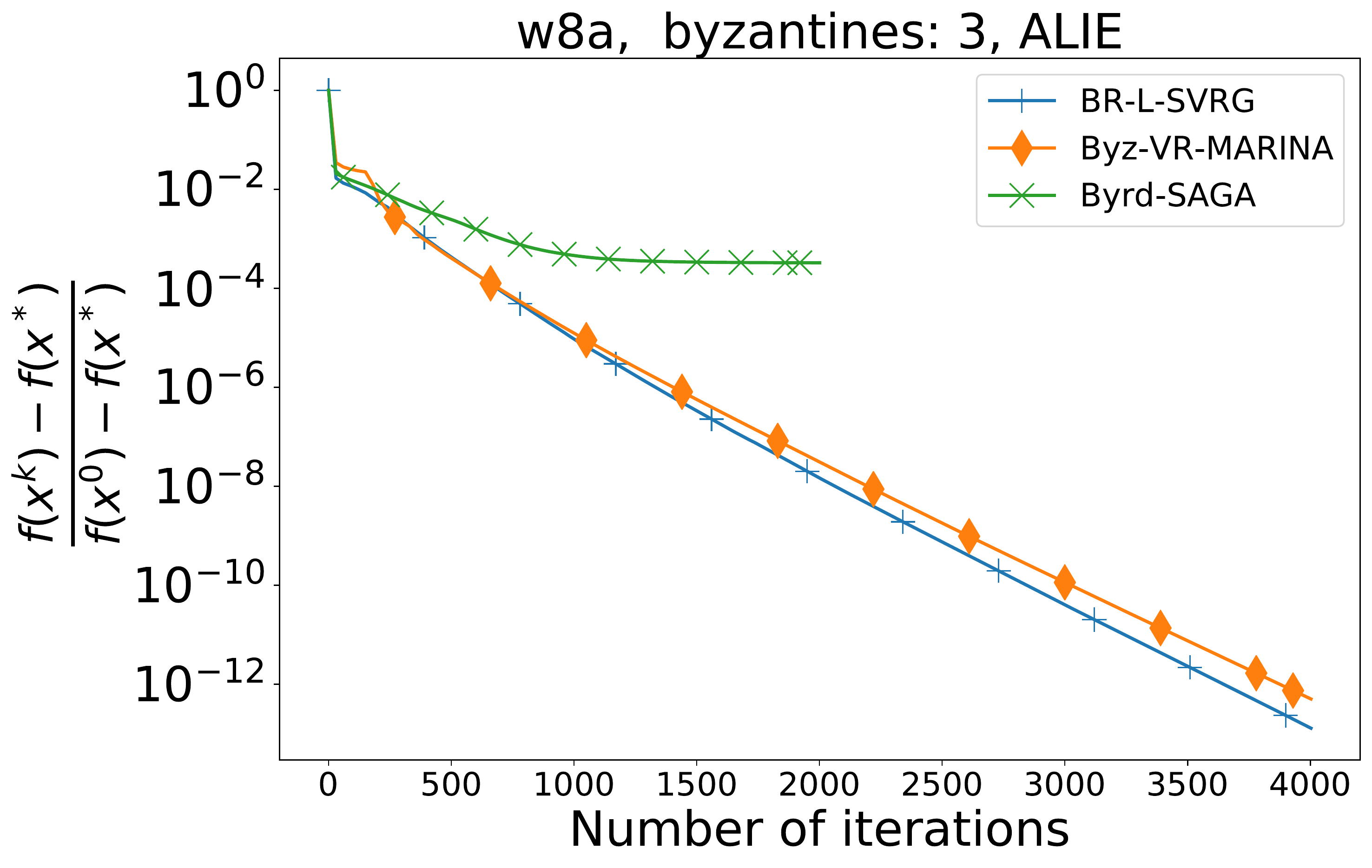}} \\
\end{minipage}
\vfill
\begin{minipage}[h]{0.24\linewidth}
\center{\includegraphics[width=1\linewidth]{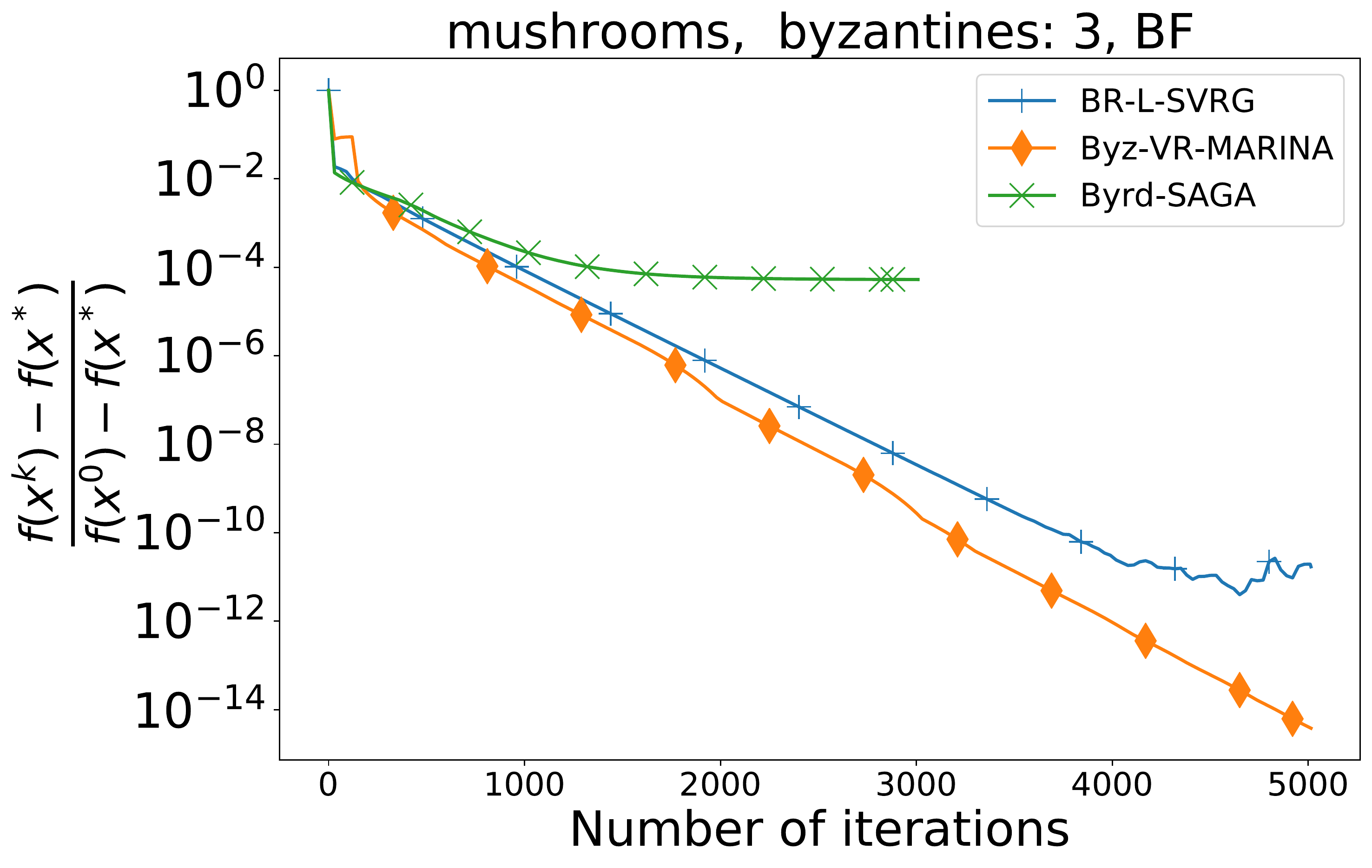}} \\
\end{minipage}
\hfill
\begin{minipage}[h]{0.24\linewidth}
\center{\includegraphics[width=1\linewidth]{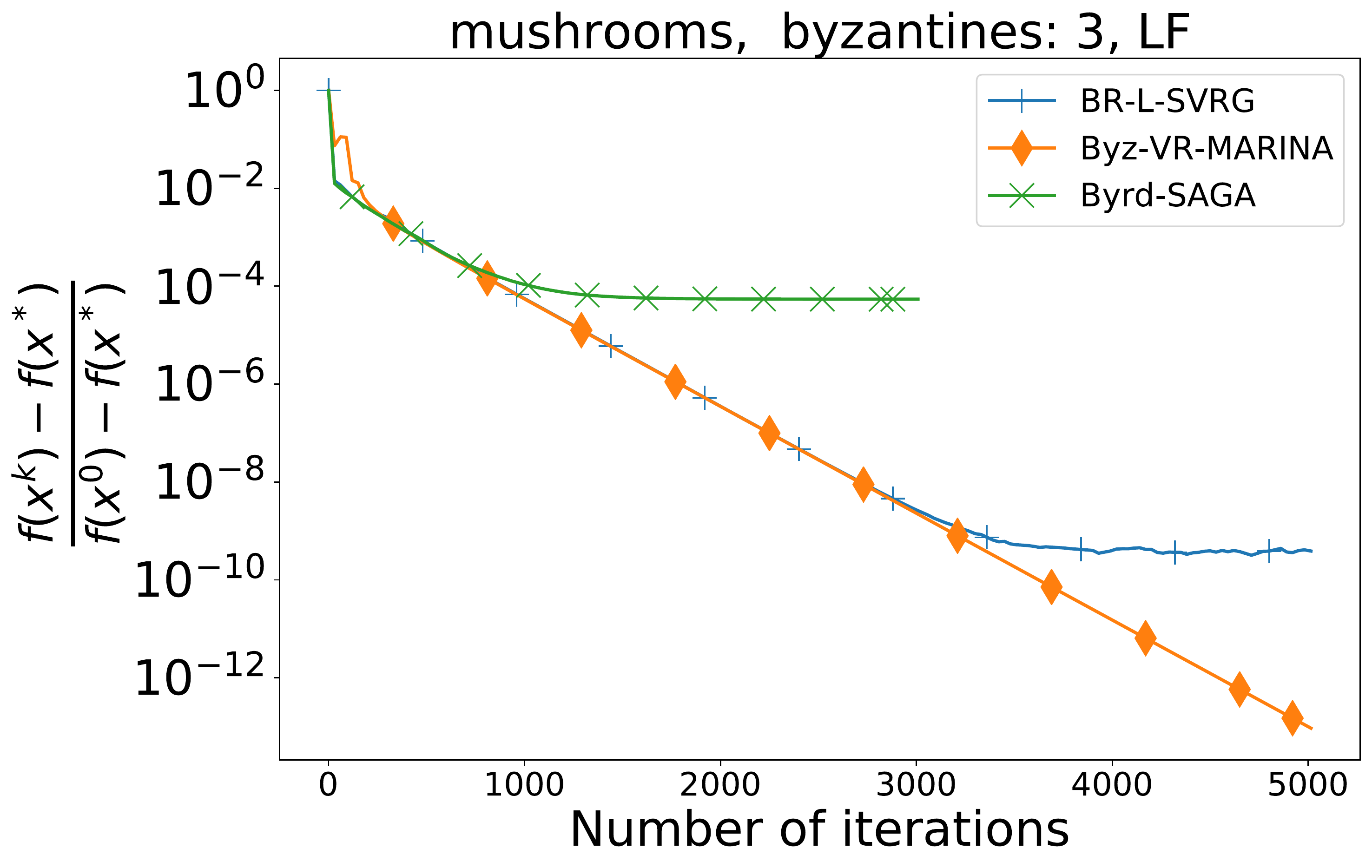}} \\
\end{minipage}
\hfill
\begin{minipage}[h]{0.24\linewidth}
\center{\includegraphics[width=1\linewidth]{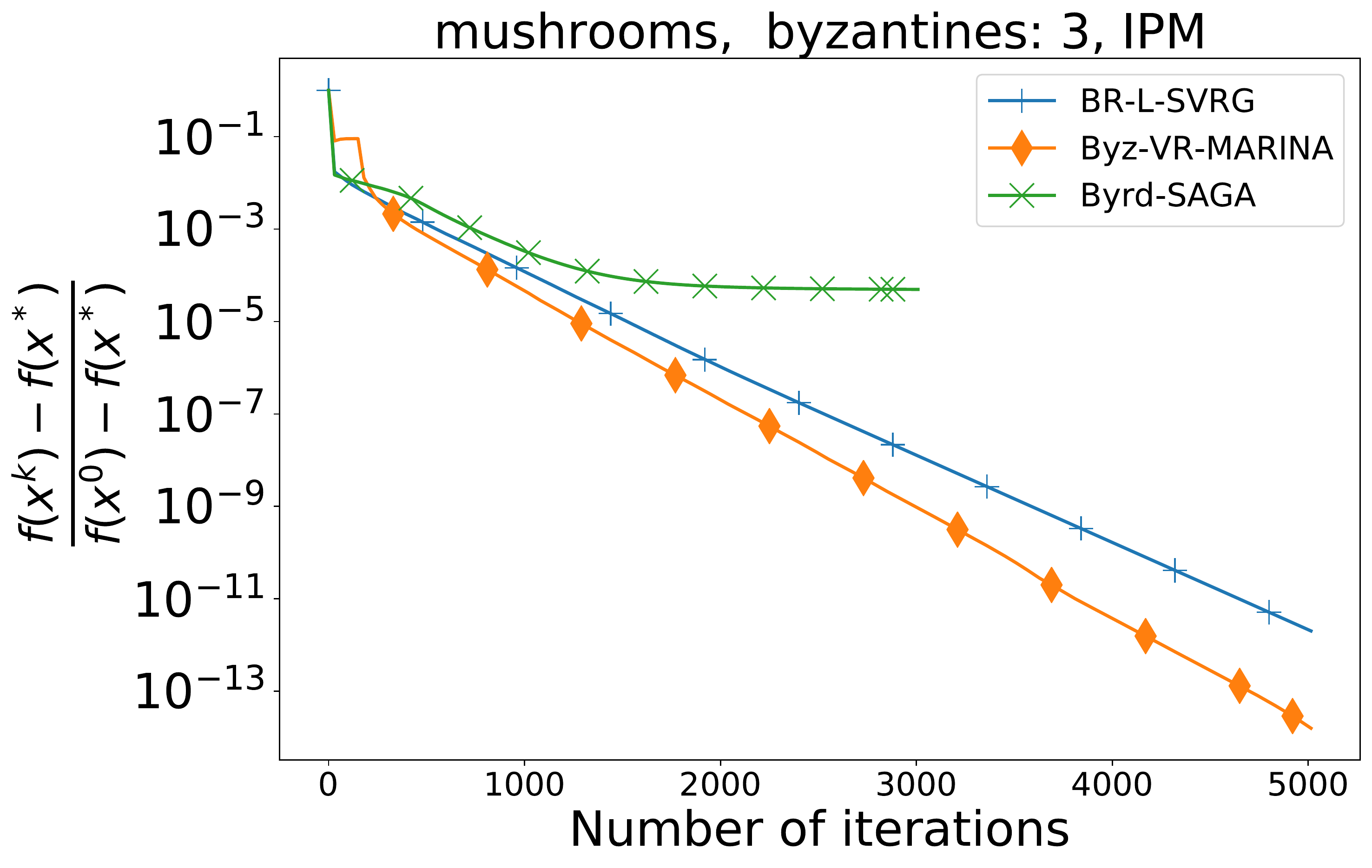}} \\
\end{minipage}
\hfill
\begin{minipage}[h]{0.24\linewidth}
\center{\includegraphics[width=1\linewidth]{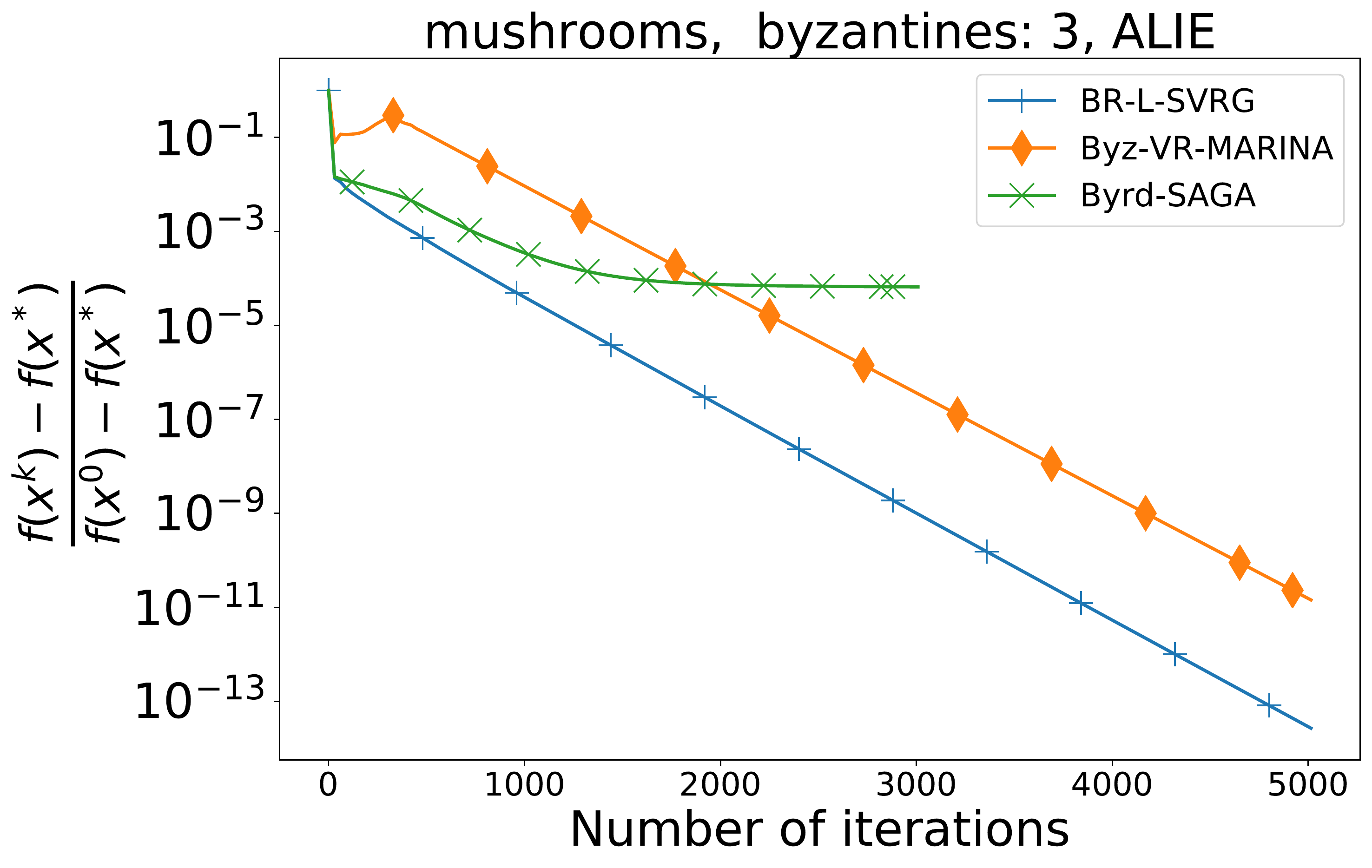}} \\
\end{minipage}
\vfill
\caption{Trajectories of \algname {BR-LSVRG}, \algname{Byz-VR-MARINA} and \algname{Byrd-SAGA} with batchsize $0.01m$. Each row corresponds to one of the used datasets with 4 different types of attacks.}
\label{second_exp}
\end{figure}

\section{Discussion}

In this work, we propose a new Byzantine-robust variance-reduced method based on \algname{SVRG}-estimator -- \algname{BR-LSVRG}. Our theoretical results show that \algname{BR-LSVRG} outperforms state-of-the-art methods in certain regimes. Numerical experiments highlight that \algname{BR-LSVRG} can have a comparable convergence to \algname{Byz-VR-MARINA}.

\bibliographystyle{splncs04}
{\bibliography{refs}}

\appendix

\section{Examples of Robust Aggregators}\label{appendix:robust_aggregation}

In \cite{karimireddy2021byzantine}, the authors propose the procedure called \emph{bucketing} (see Algorithm~\ref{alg:bucketing}) that robustifies certain aggregation rules such as: \textbullet~ geometric median (\texttt{GM}) $\hat x = \arg\min_{x\in\R^d}\sum_{i=1}^n \|x - x_i\|$; \textbullet~ coordinate-wise median (\texttt{CM}) $\hat x = \arg\min_{x\in\R^d}\sum_{i=1}^n \|x - x_i\|_1$; \textbullet~ \texttt{Krum} estimator \cite{blanchard2017machine} $\arg\min_{x_i \in \{x_1, \ldots, x_n\}} \sum_{j \in S_i} \|x_j - x_i\|^2$, where $S_i \subseteq \{x_1, \ldots, x_n\}$ is the subset of $n - |\cB| - 2$ closest (w.r.t.\ $\ell_2$-norm) vectors to $x_i$. 
\begin{algorithm}[h]
   \caption{\texttt{Bucketing}: Robust Aggregation using bucketing \cite{karimireddy2021byzantine}}\label{alg:bucketing}
\begin{algorithmic}[1]
   \STATE {\bfseries Input:} $\{x_1,\ldots,x_n\}$, $s \in \NN$ -- bucket size, \texttt{Aggr} -- aggregation rule 
   \STATE Sample random permutation $\pi = (\pi(1),\ldots, \pi(n))$ of $[n]$
   \STATE Compute $y_i = \frac{1}{s}\sum_{k = s(i-1)+1}^{\min\{si, n\}} x_{\pi(k)}$ for $i = 1, \ldots, \lceil \nicefrac{n}{s} \rceil$
   \STATE {\bfseries Return:} $\widehat x = \texttt{Aggr}(y_1, \ldots, y_{\lceil \nicefrac{n}{s} \rceil})$
\end{algorithmic}
\end{algorithm}

The following result establishes the robustness of the aforementioned aggregation rules in combination with \texttt{Bucketing}.

\begin{theorem}[Theorem D.1 from \cite{gorbunov2022variance}]
    Assume that $\{x_1,x_2,\ldots,x_n\}$ is such that there exists a subset $\cG \subseteq [n]$, $|\cG| = G \geq (1-\delta)n$ and $\sigma \ge 0$ such that $\frac{1}{G(G-1)}\sum_{i,l \in \cG}\EE\|x_i - x_l\|^2 \leq \sigma^2$. Assume that $\delta \leq \delta_{\max}$. If Algorithm~\ref{alg:bucketing} is run with $s = \lfloor \nicefrac{\delta_{\max}}{\delta} \rfloor$, then
	\begin{itemize}
	\item[\textbullet] \texttt{GM $\circ$ Bucketing} satisfies Definition~\ref{def:robust_aggr} with $c = \cO(1)$ and $\delta_{\max} < \nicefrac{1}{2}$,
	\item[\textbullet] \texttt{CM $\circ$ Bucketing} satisfies Definition~\ref{def:robust_aggr} with $c = \cO(d)$ and $\delta_{\max} < \nicefrac{1}{2}$,
        \item[\textbullet] \texttt{Krum $\circ$ Bucketing} satisfies Definition~\ref{def:robust_aggr} with $c = \cO(1)$ and $\delta_{\max} < \nicefrac{1}{4}$.
	\end{itemize}
\end{theorem}

\end{document}